\documentclass[paper=a4, fontsize=11pt]{scrartcl} 

\usepackage[T1]{fontenc} 
\usepackage{fourier} 
\usepackage[english]{babel} 
\usepackage{amsmath,amsfonts,amsthm} 
\usepackage{amsmath, calligra, mathrsfs}
\usepackage{amssymb}
\usepackage{mathtools}
\usepackage{hyperref}
\usepackage{mathdots}
\usepackage{array}
\usepackage[mathscr]{euscript}
\usepackage{verbatim}
\usepackage[dvipsnames]{xcolor}
\usepackage{mdframed} 
\usepackage{soulutf8}
\usepackage{stmaryrd}
\usepackage{tikz-cd}
\usepackage{hyperref}
\usepackage{lipsum} 

\usepackage{sectsty} 
\allsectionsfont{\centering \normalfont\scshape} 

\usepackage{fancyhdr} 
\usepackage{xurl}
\newcommand*{\sheafhom}{\mathcal{H}\kern -.5pt om}
\numberwithin{equation}{section} 
\numberwithin{figure}{section} 
\numberwithin{table}{section} 

\newtheorem{thm}{Theorem}[section]
\newtheorem{cor}[thm]{Corollary}
\newtheorem{prop}[thm]{Proposition}

\newtheorem{lem}[thm]{Lemma}

\theoremstyle{definition}
\newtheorem{defn}[thm]{Definition}

\newtheorem{exmp}[thm]{Example}

\newtheorem{notn}[thm]{Notation}

\theoremstyle{remark}
\newtheorem{rem}[thm]{Remark}

\DeclareMathOperator{\St}{st}

\DeclareMathOperator{\lk}{lk}

\DeclareMathOperator{\Susp}{Susp}

\DeclareMathOperator{\Cone}{Cone}

\DeclareMathOperator{\Span}{Span}

\DeclareMathOperator{\Ast}{ast}

\newcommand{\horrule}[1]{\rule{\linewidth}{#1}} 

\title{	
	\normalfont \normalsize 
	\textsc{} \\ [25pt] 
	\horrule{0.5pt} \\[0.4cm] 
	\huge Convex unions and completions from simplicial pseudomanifolds
	
	\horrule{2pt} \\[0.5cm] 
}

\author{Soohyun Park \\ \href{mailto:lalaland.cappelletti@gmail.com}{lalaland.cappelletti@gmail.com} } 

\date{\normalsize February 5, 2026} 

\begin{document}
	
	\maketitle 
	
	\begin{abstract}

		\noindent While intersections of convex sets are convex, their unions have rather complicated behavior. Some natural contexts where they appear include duality arguments involving boundaries of convex sets and valuations, which have an Euler characteristic-like structure. However, there are certain settings where the convexity property itself is important to consider. For example, this includes (preservation of) positivity properties of divisors on toric varieties under blowdowns. In the case of (restrictions of) conormal bundles, this can be interpreted in terms of interactions between local convexity data stored in rational equivalence relations. We consider generalizations to realizations of simplicial pseudomanifolds and replace rational equivalence with effects of PL homeomorphisms. \\
		
		\noindent Decomposing the PL homeomorphisms into edge subdivisions and contractions, we characterize the space of suitable contraction points compatible with local convexity properties in terms of convex unions and completions. This gives rise to certain external edge subdivisions that make this ``contraction space'' of the starting edge empty, which is unexpected given the expected ``increased convexity'' from edge subdivisions. We also obtain strong affine/linear restrictions on realizations of facets containing nearby edges preserving local convexity. This implies that contracting certain nearby edges results in a very large or very small contraction space of the starting edge. As for boundary behavior, there are parallels between effects of PL homeomorphisms on induced 4-cycles in the 1-skeleton. Finally, we find effects of PL homeomorphisms and suspensions on analogues of local convexity properties stored by linear systems of parameters. This indicates that simplicial spheres PL homeomorphic to the boundary of a cross polytope store record local convexity changes in the most natural way.

	\end{abstract}
	
	\section*{Introduction}
	
	Although convexity is preserved under intersections, its behavior under unions is more complicated. Some common examples where unions of convex sets appear include duality arguments converting between information on the boundary of convex sets and their support planes (p. 32, Theorem 1.6.3 on p. 33, Theorem 1.6.9 on p. 35 of \cite{Sch}) and valuations on convex sets (Chapter 6 of \cite{Sch}). In particular, the inclusion-exclusion structure in the latter case essentially allows finite unions of convex sets to be treated in a similar way to convex sets in the latter setting. However, there are natural settings where the actual convexity property of such unions is important to consider. For example, this includes positivity properties of divisors on toric varieties following a blowdown (p. 130 -- 132 of \cite{CLS}). Specializing to (restrictions of) conormal bundles on toric varieties \cite{LR}, this has connections to positivity questions on polynomials occurring in signatures of toric varieties and many different combinatorial contexts including permutation statistics \cite{Athgam}. \\

	We will study the underlying combinatorial structures from the more general perspective of realizations of simplicial pseudomanifolds. The preservation of local convexity properties under analogous transformations translates to convexity of unions of certain sets (Theorem \ref{contractpointspaceexpress}). In particular, we are working with piecewise linear (PL) homeomorphisms of simplicial pseudomanifolds. Note that the examples of topological properties involving finite unions of convex sets in the literature were mainly related to Euler characteristics (e.g.  discussion on p. 339 of \cite{Sch}). In place of specializing rational equivalence relations as in the case of toric varieties, the local convexity information is stored in inner products from rational equivalence relations and specialize to wall relations (p. 300 -- 302 of \cite{CLS}, Proposition \ref{ratwallconv}). For simplicial pseudomanifolds, we use a hyperplane (non)separation interpretation of local convexity (Proposition \ref{convsep}) and analyze the effect of PL homeomorphisms. \\

	By a result of Alexander (Corollary [10:2d] on p. 302 of \cite{Alex} and Corollary 4.1 on p. 75 of \cite{LN}), this reduces to considering edge subdivisions and contractions. In Section \ref{presconpl}, we will assume that subdividing vertices and contraction points are realized by linear interpolations of the edges involved although the reasoning involved often does not use this assumption (Remark \ref{genreasonlsop} and Remark \ref{edgesubdivconreal}). As a start, edge subdivisions always preserve convexity of existing wall crossings and introduce new convex wall crossings in this setting (Proposition \ref{edgesubdivwallconv}). Note that the individual wall crossings are not necessarily ``more convex''. \\
	
	On the other hand, edge contractions do not necessarily preserve convexity of wall crossings. For example, such edges do not even exist for boundaries of cross polytopes (Example \ref{contrcrosspolytop}). Given a fixed edge $e$ of the simplicial pseudomanifold $\Delta$ with convex wall crossings, the same reasoning shows that subdividing certain edges may produce simplicial pseudomanifolds where the contraction $e$ cannot preserve convexity of wall crossings (Theorem \ref{exsubdivcon}). This may be unexpected given the increased space of convex wall crossings coming from edge subdivisions. These are applications of our characterization of contraction points compatible with local convexity properties (Theorem \ref{contractpointspaceexpress}, Definition \ref{contptspaceconvcon}), which is the source of the connection to convex unions and completions. More generally, this space is our main tool for analyzing interactions of local convexity properties and PL homeomorphisms. \\

	In applications yielding affine/linear restrictions on realizations of facets containing ``nearby'' convex contractible edges (Definition \ref{contptspaceconvcon}), our main ideas come from general properties of convex sets and strong connectivity of simplicial pseudomanifolds. They take the form of walls with the same span inside facets containing them which are connected by wall crossings (Proposition \ref{starunintspan}). The affine/linear restrictions come from intersections of boundaries of convex sets (Proposition \ref{starunintspan}). Spreading out these restrictions uses strong connectivity of simplicial pseudomanifolds Proposition \ref{convwallspan}. Combining this with Theorem \ref{contractpointspaceexpress}, we find that contracting a nearby external edge can force extreme behavior of the starting contracting point space of an edge $e$ since it would either be at most a single point or the entire line segment realizing $e$ (Theorem \ref{convedconst}). \\
	
	There are also certain external edge subdivisions inducing such extreme behavior of the contraction point space (Theorem \ref{exsubdivcon}). In general, the outcome from external edge subdivisions depends on adjacency properties in the abstract simplicial complex. Putting this in a broader context, the restriction on nearby convex contractible edges fits into the effects of external edge contractions on contraction point spaces. Unlike external edge subdivisions, there is a variation in behavior even under same adjacency properties in the abstract simplicial complex structure (Corollary \ref{abedgeconlin}). This depends on whether the hyperplanes bounding the half-spaces determining local convexity properties separate the endpoints of the edge contracted and the size of the projection to these hyperplanes (Proposition \ref{genextcon}). Normalizing for length of vectors realizing vertices of the contracted edge can lead one to interpret this as a comparison of angles made with these hyperplanes. \\

	Motivated by combinatorial heuristics and the discussion above, we consider how boundary cases of flat wall crossings interact with suspensions and PL homeomorphisms. More specifically, we find that there are parallels between effects of PL homeomorphisms on 4-cycles in the 1-skeleton and on flat wall crossings induced by changes in the abstract simplicial complex structure (Proposition \ref{bdry4cycleflat}). For changes purely from $k^d$-realizations, we find differences in behavior arising from ``degeneracies'' from containment in hyperplanes bounding half-spaces for convexity conditions. The combinatorial heuristics come from close connections between flag simplicial complexes and local convexity in combinatorial positivity questions coming from flag simplicial complexes (\cite{LR}, \cite{Psig}, \cite{Athgam}). Note that flagness is similar to the condition for contractions of edges to preserve the PL homeomorphism class \cite{Nev}. Edge subdivisions preserve flagness and existing convex wall crossings while increasing new ones. The cases that are minimal with respect to edge subdivisions are those whose 1-skeleta are covered by induced 4-cycles. Applying a similar ``minimal'' idea led to the comparison to flat wall crossings. Heuristically, we can think about these cases as building blocks where others constructed out of edge subdivisions of them. \\

	Apart from connections to intrinsic combinatorial data from simplicial complexes, we recall that the reasoning for the main results often do not depend on linear interpolations (Remark \ref{genreasonlsop} and Remark \ref{edgesubdivconreal}). This leads us to analogous properties for more general realizations of simplicial pseudomanifolds. In the course of doing this, we find compatibility of linear interpolations with such properties and further generalizations (Remark \ref{edgesubdivconreal}, Proposition \ref{flaglsoptrace}). In place of rational equivalence, we work with linear systems of parameters. They give linear relations between coordinate realizations (e.g. center of mass relations or stress spaces in Remark \ref{lsopptconfig} and Remark \ref{lsopwallcoord}). For example, this includes rational equivalence relations of toric varieties. These linear forms also serve as a transition between combinatorial structures of abstract simplicial complexes and generalizations of Chow rings of toric varieties. Even with additional assumptions on proper colorings of the 1-skeleton, the latter does not give sufficient information on how the realizations of the faces fit together. \\

	Using an appropriate choice of linear systems of parameters, we can obtain linear forms that are closely connected to wall relations  (Example \ref{crosspolylsop}, Remark \ref{wallanaloguelsop}). We construct natural choice of linear systems of parameters that track changes from PL homeomorphisms (Proposition \ref{lsopsuspedgediv} and Proposition \ref{edgecongenlsop}). The coefficient signs involved is an application of observations at the end of Remark \ref{edgesubdivconreal}. We give an example illustrating compatibility with linear interpolations (Example \ref{edgelsopavg}). Combining components of this framework, we find effects of PL homeomorphisms and suspensions on analogues of local convexity properties stored by linear systems of parameters (Proposition \ref{flaglsoptrace}). In some sense, this indicates that simplicial spheres PL homeomorphic to boundary of a cross polytope record local convexity changes in the most ``pure'' way. Finally, the same reasoning yields properties analogous to \emph{changes} in local convexity for more general classes of simplicial pseudomanifolds (Corollary \ref{genplhomeom}). \\

	The convexity and simplicial pseudomanifold for the properties discussed above are described in Section \ref{backnot}. Afterwards, we give a hyperplane (non)separation interpretation of local convexity applicable to realizations to simplicial pseudomanifolds and define the support in Section \ref{hypsepconv}. This is the main tool of our analysis of effects of interactions between local convexity properties of realizations of simplicial pseudomanifolds and PL homeomorphisms in Section \ref{presconpl}. Finally, we connect replacements of the rational equivalence relation for toric varieties for simplicial pseudomanifolds to analogues of local convexity in Section \ref{lsopcoordrat}. This involves extending our earlier observations on more general coordinate realizations.

	\section{Background and notation} \label{backnot}

	We give some initial background on main tools to analyze (local) convexity properties of simplicial pseudomanifold realizations in Section \ref{presconpl} and Section \ref{lsopcoordrat} using the hyperplane (non)separation interpretation of convexity in Section \ref{hypsepconv}. \\
	
	Before discussing other properties, we define the main geometric property we are interested in. \\

	\begin{defn} \textbf{(Convex sets) }
		\begin{enumerate}
			\item A set $A \subset k^d$ is \textbf{convex} if $\alpha, \beta \in A$ implies that any point $(1 - \eta) \alpha + \eta \beta$ with $\eta \in [0, 1]$ on the line segment $\overline{\alpha \beta}$ is contained in $A$. 
			
			\item The ``nondegenerate'' convex sets (i.e. not affine spaces or half-spaces) can be expressed both as convex linear combinations of points on the boundary and intersections of half-spaces containing them (Corollary 1.3.5 on p. 12 and Lemma 1.4.1 on p. 16 of \cite{Sch}, Proposition 1.2.8 on p. 25 and Proposition 1.2.10 on p. 26 -- 27 of \cite{CLS}). \\
		\end{enumerate}
	\end{defn}

	We are primarily concerned with sets that are realizations of simplicial complexes. Taking a page from toric varieties, we will define this to be the support of a simplicial complex. \\

	\begin{defn} \textbf{(Support of a simplicial complex) \\} (generalizing Definition 3.1.2 on p. 106 of \cite{CLS}) \\
		The \textbf{support} $|S|$ of a $(d - 1)$-dimensional simplicial complex $S$ is a fixed realization of $S$ in $k^d$. This is the union of realizations of faces of $S$. It is based on terminology of (polyhedral) fans (e.g. in the context of toric varieties). \\
	\end{defn}
	
	\begin{rem} \textbf{(Linear independence of vertices on a face) \\}
		Since faces of simplicial complexes intersect on faces, nonempty intersections of distinct faces must lie on the boundaries of the respective faces. Given this condition, the realizations of vertices of a face of a simplicial complex will be taken to be linearly independent. \\
	\end{rem}

	\begin{notn} \textbf{(Containment in faces of simplicial complexes) \\}
		If $p$ is a vertex of $\Delta$, we will interpret the containment $p \in F$ for a face $F \in \Delta$ terms of containment of vertex sets. However, a point $\alpha$ of $k^d$ that is \emph{not} the realization of a vertex of $F$, we will interpret $\alpha \in F$ to mean the containment of $\alpha$ in the \emph{realization} of $F$ as a point of $k^d$.
	\end{notn}

	The particular simplicial complexes we are interested in are simplicial pseudomanifolds, which we define below. Since wall crossing-related arguments are important for our analysis of local convexity properties, they will also be included in the definition. \\
	
	\begin{defn} \textbf{(Simplicial pseudomanifolds and wall crossings) \\ } (p. 46 - 48 of \cite{Fri}, p. 90 - 91 from Section 24, p. 86 - 87 from Section 23, and p. 152 from Section 41 of \cite{SW}, p. 24 of \cite{St}, p. 20 and p. 37 of \cite{Hud}) \\ \label{pseudomfldwallcross}
		
		\begin{enumerate}
			
			\item A $(d - 1)$-dimensional \textbf{simplicial pseudomanifold without boundary} $\Delta$ satisfies the following properties:
			\begin{enumerate}
				\item Every face of $\Delta$ is a face of a $(d - 1)$-dimensional face of $\Delta$.
				
				\item Every $(d - 2)$-dimensional face of $\Delta$ is the face of exactly two $(d - 1)$-dimensional faces of $\Delta$.
				
				\item If $F$ and $F'$ are $(d - 1)$-dimensional faces of $\Delta$, there is a finite sequence $F = F_1, \ldots, F_m = F'$ of $(d - 1)$-dimensional faces of $\Delta$ such that $F_i \cap F_{i + 1}$ is a $(d - 2)$-dimensional face of $\Delta$. 
			\end{enumerate}
			
			In our setting, we will simply call this a \textbf{simplicial pseudomanifold. \\}

			\item In the context of toric varieties, such codimension 1 faces are called \textbf{walls}. For example, they appear in the context of positivity of divisors on toric varieties and convexity-related properties (e.g. p. 265, p. 301 -- 302 of \cite{CLS}, p. 419 -- 422 of \cite{Mat}). Since we will use convexity arguments involving facet traversals in realizations of simplicial pseudomanifolds, we will also call facet traversals on simplicial pseudomanifolds \textbf{wall crossings}. \\
		\end{enumerate}
	\end{defn}
	
	\begin{rem} \textbf{(Further comments on simplicial pseudomanifolds) \\} \label{simppseudcom}
		A discussion on the definition of a simplicial pseudomanifold is on p. 46 -- 48 of \cite{Fri}. For example, it includes possible lower-dimensional obstructions from being an actual (triangulable) manifold. Note that the original definition only included the first two conditions. The topological motivation for the third condition (strong connectivity) is discussed on p. 47 of \cite{Fri} with further details on p. 90 - 91 from Section 24, p. 86 - 87 from Section 23, and p. 152 from Section 41 of \cite{SW}. A more ``intrinsic'' description of the second condition comes from links over such coming from $S^1$ (p. 26 of \cite{Hud}). Note that there are differences in whether the third part is assumed in different contexts (e.g. compare p. 24 of \cite{St} and p. 20 of \cite{Athsom} with p. 200 of \cite{NS}). \\
	\end{rem}
	
	Wall crossings have a natural connection to local convexity properties and the convex sets we will be ``tiling'' the support with (Proposition \ref{convsep}). In the context of toric varieties, local convexity data has applications to positivity properties of divisors (e.g. support functions of basepoint free and ample divisors in Theorem 6.1.7 on p. 266 -- 267 and Theorem 6.1.14 on p. 271 of \cite{CLS}) and combinatorial properties involving them (e.g. signatures of toric varieties in \cite{LR}). The result below gives a concrete relationship between wall crossings and local convexity properties. \\
	
	\begin{prop} \textbf{(Rational equivalence, wall relations, and convexity) \\} (Lemma 14-1-7 on p. 421 of \cite{Mat}) \label{ratwallconv} \\ 
		Let $\Sigma$ be a $(d - 1)$-dimensional simplicial fan lying inside $k^d$. Consider a wall crossing in a fan $\Sigma$ using the facets $F$ and $G$. Fix an on-wall ray $\rho \in \Sigma(1)$. Let $p \in F$ and $q \in G$ be the off-wall rays and $\tau \coloneq F \cap G$ be a wall containing $\rho$. The \textbf{wall relation} \[ c_p x_p + c_q x_q + \sum_{m \in \tau(1)} c_m x_m = 0 \] with $c_p > 0$ with $c_q > 0$ associated to this wall crossing is a linear relation among vertices of $F \cup G$ in $k^d$, which is unique up to scaling (p. 301 -- 302 of \cite{CLS}). \\
		
		Noting that $\rho \in \tau(1)$, we have the following equivalences using dual bases to the rays of $F$.
		
		\begin{enumerate}
			\item $D_\rho \cdot V_\tau > 0 \Longleftrightarrow \langle \rho_F^*, q \rangle < 0 \Longleftrightarrow c_\rho > 0 \Longleftrightarrow F \cup G$ is \textbf{strictly concave} along $\tau \setminus \rho$/with respect to $\rho$.
			
			\item $D_\rho \cdot V_\tau = 0 \Longleftrightarrow \langle \rho_F^*, q \rangle = 0 \Longleftrightarrow c_\rho = 0 \Longleftrightarrow F \cup G$ is \textbf{flat} along $\tau \setminus \rho$/with respect to $\rho$.
			
			\item $D_\rho \cdot V_\tau < 0 \Longleftrightarrow \langle \rho_F^*, q \rangle > 0 \Longleftrightarrow c_\rho < 0 \Longleftrightarrow F \cup G$ is \textbf{strictly convex} along $\tau \setminus \rho$/with respect to $\rho$.
		\end{enumerate}
	\end{prop}

	The wall relations (see p. 301 -- 302 of \cite{CLS}) involved in the proof of this result give a transition between purely combinatorial data and algebraic structures. In Proposition \ref{convsep}, we will give a more geometric interpretation of Part 3 in terms of hyperplane (non)separation and its connection to local convexity properties. Note that this can be rephrased for the other parts as well. \\
	
	\section{(Non)separation by hyperplanes and convexity} \label{hypsepconv}

	Since we are working with convexity properties depending on coordinate realizations of simplicial complexes, we first define related terms.

	\begin{defn} (Definition 3.1.2 on p. 106 of \cite{CLS} and p. 259 of \cite{LR} applied to simplicial complexes) \\
		Consider a $(d - 1)$-dimensional simplicial complex $\Delta$ with a fixed realization in $k^d$. 
		\begin{enumerate}
			\item The \textbf{support} $|\Delta|$ of $\Delta$ is the union of the realizations of the facets of $\Delta$.
			
			\item The simplicial complex $\Delta$ is \textbf{locally convex} if $|\St_\Delta(p)|$ is convex for all vertices $p \in V(\Delta)$. 
		\end{enumerate}
	\end{defn}

	Note that the local convexity property has close connections to combinatorial positivity properties lying between unimodality and real-rootedness and nef/basepoint free properties of (restrictions of) conormal bundles of toric varieties (\cite{LR}, \cite{Psig}, \cite{Athgam}). \\
	
	Next, we give a hyperplane (non)separation interpretation of Proposition \ref{ratwallconv} followed by implications for local and global convexity properties. This will be the perspective we take while analyzing local convexity properties of realizations of simplicial pseudomanifolds. \\
	
	\color{black} 
	
	\begin{prop} \textbf{(Hyperplane (non)separation and convexity properties) \\} \label{convsep}
		\begin{enumerate}
			\item A wall crossing with a wall $\tau$ and off-wall vertices $r$ and $s$ is convex with respect to a vertex $p \in \tau$ on the wall $\tau$ (in the sense of p. 421 of \cite{Mat}) if and only if the hyperplane $\Span(\tau \setminus p, r)$ does \emph{not} separate the points $p$ and $s$. Note that this is equivalent to the hyperplane $\Span(\tau \setminus p, s)$ does \emph{not} separate the points $p$ and $r$. The wall crossing is flat with respect to $p$ when $s \in \Span(\tau \setminus p, r)$. \\
			
			\item The union $F \cup G$ of two facets $F, G \in \Delta$ sharing a wall $F \cap G$ is convex if and only if the associated wall crossing is convex with respect to every on-wall vertex $p \in F \cap G$. \\

			\item The support of the star $|\St_\Delta(p)|$ is convex if and only if all the wall crossings in $\Delta$ containing $p$ on the wall are convex with respect to $p$. \\
		\end{enumerate}
	\end{prop}
	
	\begin{rem}
		Part 3 of Proposition \ref{convsep} is a combinatorial counterpart of the equivalence of the nefness of a Cartier divisor on a toric variety associated to a fan with full-dimensional convex support and nonnegative intersection with torus-invariant irreducible complete curves (Theorem 6.3.12 on p. 291 of \cite{CLS} and Lemma 14-1-7 on p. 421 of \cite{Mat}). \\
	\end{rem}

	\begin{proof}
		\begin{enumerate}
			\item This is a non-separating hyperplane interpretation of Proposition \ref{ratwallconv}. \\
			
			\item Let $p \in F$ and $q \in G$ be the off-wall vertices of the wall crossing associated to $F \cup G$. By Lemma 1.4.1 on p. 16 of \cite{Sch}, a closed convex set is the convex hull of its boundary unless it is an affine subspace of the ambient space $k^d$ or a half-space. Thus, the set $F \cup G$ is convex if and only if it is the convex hull of its boundary. In order for $F \cup G$ to be the convex hull of its boundary, the entire set $F \cup G$ can only lie on one side of each hyperplane spanned by vertices of any wall of $F$ or $G$ except $F \cap G$ (i.e. any wall of $F$ containing $p$ or any wall of $G$ containing $q$). More specifically, the convexity of $F$ implies that a hyperplane of the form $\Span(F \setminus m)$ for some $m \in F \setminus p$ does \emph{not} separate the points $p$ and $r$ for any $r \in F \setminus \{ p, m \}$. The same statement holds with $G$ replacing $F$ and $q$ replacing $p$. \\ 
			
			Since the vertices of $F$ always lie on the same side of $\Span(F \setminus m)$ for some $m \ne p$, the only possible ``problem vertex'' is $q$. Alternatively, we can phrase this the question of whether a hyperplane of the form $\Span(G \setminus m)$ for some on-wall vertex $m \in F \cap G$ separates $p$ and $q$. Consider when it does \emph{not} separate $p$ and $q$. If $p \in \Span(G \setminus m)$, the hyperplane $\Span(G \setminus m)$ does not separate $p$ and $q$ since $p$ lies on both half-spaces associated to the hyperplane. This is equivalent to the coefficient of $m$ in the wall relation associated to the wall crossing for $F \cup G$ being equal to $0$. Suppose that $p \notin \Span(G \setminus m)$. This would imply that $\overline{pq} \not\subset \Span(G \setminus m)$ and the intersection $\overline{pq} \cap \Span(G \setminus m)$ is a single point or the empty set. Since $q$ already lies in this intersection, we have that $\overline{pq} \cap \Span(G \setminus m) = \{ q \}$ and the hyperplane $\Span(G \setminus m)$ does \emph{not} separate $p$ and $q$. The convexity of half-spaces then implies that the entire line segment $\overline{pq}$ lies on one side of the hyperplane $\Span(G \setminus m)$. The convexity of the wall crossing for $F \cup G$ with respect to $m$ is the statement that $p$ and $m$ lie on the same side of the hyperplane $\Span(G \setminus m)$). This follows from the statement of Part 1. Repeating this for each $m \in F \cap G$ yields the stated equivalence. The same argument can be repeated with with $p$ replacing $q$ and the hyperplane $\Span(F \setminus m)$ replacing $\Span(G \setminus m)$. \\

			Alternatively, we can approach this more directly using wall relations. Note that $F$ and $G$ are convex linear combinations of their respective vertices. Since $p$ and $q$ do \emph{not} lie inside the hyperplane $\Span(F \cap G)$, the intersection $\overline{pq} \cap \Span(F \cap G)$ is a single point. In order for $F \cup G$ to be convex, this point must lie inside $F \cup G$. The linear independence of the realizations of the vertices in $F \cap G \in \Delta$ in $k^d$ implies that the point must lie in the wall $F \cap G$. Equivalently, we have that \[ (1 - \eta) x_p + \eta x_q = \sum_{r \in F \cap G} c_r x_r \] for some $\eta \in (0, 1)$ and $c_r \ge 0$ such that $\sum_{r \in F \cap G} c_r = 1$ (i.e. a convex linear combination). This can be rearranged into the relation \[ (1 - \eta) x_p + \eta x_q + \sum_{r \in F \cap G} (-c_r) x_r = 0. \] This can be interpreted as the coefficients of the on-wall vertices $r \in F \cap G$ being nonpositive in the wall relation associated to the wall crossing for $F \cup G$, which is the condition for convexity of this wall crossing with respect to each on-wall vertex $r \in F \cup G$ (see Lemma 14-1-7 on p. 421 of \cite{Mat}). \\

			\item By the boundary convex hull argument in Part 2, the convexity of $|\St_\Delta(p)|$ would imply that each wall crossing between facets of $\Delta$ containing $p$ is convex with respect to $p$. Now suppose that all wall crossings containing $p$ on the wall are convex with respect to $p$. Given each span of a wall opposite to $p$, the off-wall vertex is on the same side of this hyperplane as $p$. This means that $p$ lies in the intersection of all the half-spaces we would like to intersect in order to form the support $|\St_\Delta(p)|$. Suppose there is a vertex $m \in \lk_\Delta(p)$ that is on the opposite side of $p$ with respect to one of these hyperplanes. This would form an edge that intersects the interior of a face of $\Delta$ (e.g. an edge or facet of $\St_\Delta(p)$), which is not possible for a pair of faces of a simplicial complex.
			
		\end{enumerate}
	\end{proof}
	
	\section{Preservation of local convexity, wall crossings, and PL homeomorphisms} \label{presconpl}

	This is the section with our main results on interactions between convexity properties of realizations of simplicial pseudomanifolds and PL homeomorphisms. After initial simplifications of the problem and studying the behavior under edge subdivisions, we will find the space characterizing preservation of convex wall crossings in Section \ref{plconvcom}. The main applications of this space to possible realizations, proximity of convex contractible edges, and range of possible behavior of this space are in Section \ref{convpresrest}. Finally, we take a closer look at comparisons underlying combinatorial properties from the underlying abstract simplicial complex structure via boundary conditions in Section \ref{bdrybeh}.
	
	\subsection{PL homeomorphisms and convex unions/completions} \label{plconvcom}

	As mentioned above, we find the space characterizing preservation of local convexity properties under PL homeomorphisms (Theorem \ref{contractpointspaceexpress}) in this subsection. We start by reducing our analysis on effects of PL homeomorphisms to edge subdivisions and contractions. \\
	
	\color{black}

	\begin{thm} (Alexander, Corollary [10:2d] on p. 302 of \cite{Alex} and Corollary 4.1 on p. 75 of \cite{LN}) \\
		If $\Delta$ and $\Gamma$ are PL homeomorphic simplicial complexes, then they are connected by a sequence of edge subdivisions and contractions. \\
	\end{thm}
	
	Since changes induced by (stellar) subdivisions and contractions of edges are induced by local data from links over edges, the changes are associated to \emph{local} data from links over faces. We will connect this back to local convexity data via natural changes in linear systems of parameters and $k^d$-realizations arising from PL homeomorphisms. While we will assume that subdividing vertices and contraction points of edges are realized by linear interpolations, we note that a large part of the reasoning applies to more general realizations (Remark \ref{genreasonlsop} and Remark \ref{edgesubdivconreal}). \\

	As a start, we show that edge subdivisions with the subdividing vertex realized by a linear interpolation increase the space of convex wall crossings. We note that they do not necessarily make individual wall crossings ``more convex''. This is something we will keep in mind for effects of PL homeomorphisms involving external edges. While we will use linear interpolations for realizations of subdividing vertices and contraction points, the reasoning involved often does \emph{not} depend on the realizations used. Further discussions on realizations are also in Remark \ref{genreasonlsop} from Section \ref{lsopcoordrat}. \\
	
	\color{black}
	
	\begin{prop} \textbf{(Edge subdivisions and (local) convexity) \\} \label{edgesubdivwallconv}
		Let $\Delta'$ be the (stellar) subdivision of a $(d - 1)$-dimensional Cohen--Macaulay simplicial complex $\Delta$ with respect to an edge $e = \{ a, b \} \in \Delta$. Let $v \in \Delta'$ be the subdividing vertex. Let $x_p$ be the coordinate representation of the point $p \in V(\Delta)$ in $k^d$. Suppose that $x_v = (1 - \eta) x_a + \eta x_b$ for some $\eta \in (0, 1)$ and $x_r$ for $r \in V(\Delta)$ is the same in $\Delta'$ as $\Delta$. \\
		
		If the starting simplicial complex $\Delta$ has convex wall crossings, the wall crossings of $\Delta'$ are also convex. In particular, the wall crossings using $v$ as an on-wall vertex are ``at least as convex'' as those of $\Delta$ in the sense that previously convex wall relations stay convex (or become more convex) and newly introduced wall crossings are convex (Proposition \ref{ratwallconv} and Proposition \ref{convsep}). While this is not necessarily true for wall crossings in $\Delta'$ using $v$ as an off-wall vertex, the induced coefficient changes do \emph{not} introduce any new flat wall crossings. \\
	\end{prop}

	\begin{proof}
		We will keep the $k^d$-coordinates of ``old'' vertices from $\Delta$ to $\Delta'$ the same as their ``original'' coordinates. These are all the vertices of $\Delta'$ except the subdividing vertex $v$. \\

		We will first focus on the wall crossings involving the subdividing vertex $v$ that are \emph{not} counterparts of those in $\Delta$. We can split such wall crossings into those that use $v$ as an on-wall vertex and those that use $v$ as an off-wall vertex. The facets of $\Delta'$ containing $v$ correspond to facets of $\Delta$ containing $e = \{ a, b \}$ with one of $a$ or $b$ replaced by $v$. If we consider wall crossings with $a$ replaced by $b$ or $b$ replaced by $a$, the chosen relation $x_v = (1 - \eta) x_a + \eta x_b$ for a fixed $\eta \in (0, 1)$ yields the wall relation \[ (1 - \eta) x_a + \eta x_b - x_v = 0. \] Since the coefficient of $x_v$ is $-1$ and the $x_p$ for the remaining on-wall vertices $p$ have coefficient 0, we have a convex wall crossing in that case (see signs of on-wall vertices on p. 420 -- 421 of \cite{Mat}). \\
		
		Next, we will consider wall crossings in $\Delta'$ using $v$ that are counterparts of those in $\Delta$. Suppose that the starting off-wall vertex of the facet of $\Delta'$ containing $v$ is a vertex $q \in V(\Delta)$ such that $q \ne a, b$. This means that $v$ would be an on-wall vertex. Note that $\lk_{\Delta'}(v) = \Susp_{a, b}(\lk_\Delta(e))$. This would be a counterpart of a wall crossing in $\Delta$ involving 2 facets that both contain $e$. The difference is that one of the two on-wall vertices $a$ or $b$ is replaced by $v$ in $\Delta'$. Suppose that $b$ is replaced with $v$. Since the $k^d$-coordinates of the ``old'' vertices of $\Delta$ (i.e. vertices of $\Delta'$ other than $v$) are assumed to be the same for $\Delta'$, then the wall relation for this wall relation in $\Delta$ with the wall denoted by $\tau$ is \[ c_q x_q + c_{q'} x_{q'} + \sum_{m \in \tau} c_m x_m = 0 \] with $c_q, c_{q'} > 0$ and $c_m \le 0$ for all $m \in \tau$ (assuming $\tau \supset e = \{ a, b \}$). Its counterpart in $\Delta'$ has the wall relation

		\[ c_q x_q + c_{q'} x_{q'} + \sum_{ \substack{ m \in \tau \\ m \ne b } } c_m x_m + c_b \left( \frac{1}{\eta} x_v - \frac{1 - \eta}{\eta} x_a \right) = c_q x_q + c_{q'} x_{q'} + \sum_{ \substack{ m \in \tau \\ m \ne a, b } } c_m x_m + \left( c_a - \frac{1 - \eta}{\eta} \right) x_a + \frac{c_b}{\eta} x_v
		= 0 \]
		
		since \[ x_b = \frac{1}{\eta} (x_v - (1 - \eta) x_a) = \frac{1}{\eta} x_v - \frac{1 - \eta}{\eta} x_a. \]
		
		Note that $c_a - \frac{1 - \eta}{\eta} < 0$ since $c_a \le 0$ and $\frac{1 - \eta}{\eta} > 0$. It is clear that $\frac{c_b}{\eta} \le 0$ since $c_b \le 0$ and $\eta > 0$. Since $0 < \eta < 1$, we actually have that $\frac{c_b}{\eta} \le c_b \le 0$. Thus, the wall relation \[ c_q x_q + c_{q'} x_{q'} + \sum_{ \substack{ m \in \tau \\ m \ne a, b } } c_m x_m + \left( c_a - \frac{1 - \eta}{\eta} \right) x_a + \frac{c_b}{\eta} x_v = 0 \] in $\Delta'$ has strictly positive coefficients for the off-wall vertices and nonpositive coefficients for the on-wall vertices. This means that the associated wall crossing in $\Delta'$ is convex in the sense of Proposition \ref{ratwallconv} and Proposition \ref{convsep}. Note that it is ``even more convex'' than the corresponding wall crossing in $\Delta$ since all the coefficients of the off-wall vertices stay the same and the coefficients of the on-wall rays either stay the same or become even more negative. All of this can be repeated when the starting facet of $\Delta'$ containing $v$ is one that replaces $a$ by $v$ in a facet of $\Delta$ containing $e = \{ a, b \}$. \\
		
		Now consider wall crossings in $\Delta'$ starting with a facet containing $v$ that uses $v$ as an off-wall vertex. This means that the vertex $v$ is replaced by another vertex. Without loss of generality, suppose that this facet comes from replacing $b$ by $v$ in a facet of $\Delta$ containing $e = \{ a, b \}$. Then, we have that the vertex that $v$ is replaced by is the vertex $b' \in V(\Delta)$ that $b$ is replaced by in the counterpart of this wall crossing in $\Delta$. For this latter wall crossing in $\Delta$, we have the wall relation \[ c_b x_b + c_{b'} x_{b'} + \sum_{m \in \tau} c_m x_m = 0 \] with $c_b, c_{b'} > 0$ and $c_m \le 0$ for all $m \in \tau$. Since the chosen relation $x_v = (1 - \eta) x_a + \eta x_b$ implies that $x_b = \frac{1}{\eta} x_v - \frac{1 - \eta}{\eta} x_a$, we have the wall relation \[ c_b \left( \frac{1}{\eta} x_v - \frac{1 - \eta}{\eta} x_a \right) + c_{b'} x_{b'} + \sum_{m \in \tau} c_m x_m = \frac{c_b}{\eta} x_v + c_{b'} x_{b'} + \sum_{ \substack{m \in \tau \\ m \ne a} } c_m x_m + \left( c_a - \frac{1 - \eta}{\eta} c_b \right) x_a = 0 \] in $\Delta'$. Since $0 < \eta < 1$ and $c_b > 0$, we have that $\frac{c_b}{\eta} > c_b > 0$ and $\frac{1 - \eta}{\eta} > 0$. Combining this with $c_a \le 0$, we have that $c_a - \frac{1 - \eta}{\eta} c_b < c_a \le 0$. This implies that the wall crossing we are considering is convex. Again, the reasoning can be repeated when the off-wall vertex $v$ of the starting facet of the wall crossing is taken to replace $a$ by $v$ in a facet of $\Delta$ containing $e = \{ a, b \}$. Since $0 < c_b < \frac{c_b}{\eta}$, the wall crossings are not necessarily ``more convex'' than their counterparts in $\Delta$. However, the decrease in ``unscaled'' coordinates $c_r$ indicates that there are no new flat wall crossings generated.

	\end{proof}

	\begin{defn} \label{contptspaceconvcon}
		Let $\Delta$ be a $(d - 1)$-dimensional simplicial complex with a fixed realization in $k^d$. Suppose that $\Delta$ has convex wall crossings. Consider an edge $e = \{ a, b \} \in \Delta$. The \textbf{contraction point space} of $e$ is the space of possible points $w \in k^d$ replacing $a$ and $b$ such that the resulting contraction of $e$ has convex wall crossings. The edge $e$ is \textbf{convex contractible} if such a point $w$ exists. \\
	\end{defn}

	\color{black}

	The idea for contraction point spaces is that we are looking for suitable contraction points ``completing'' a certain collection of convex sets to larger convex ones compatible with realizations of simplicial complexes. Our first main result states that these are given convex sets whose union with certain convex sets storing local data are convex. It will be expressed in terms of half-spaces. A key input is the hyperplane (non)separation interpretations of convexity from Proposition \ref{convsep}. In particular, we do \emph{not} want separation by designated hyperplanes from centers of stars of vertices we measure convexity with respect to. This yields possible ``completions'' of sets required to be convex in order for local convexity to be satisfied. \\
	
	\color{black}

	\begin{thm} \textbf{(Space of possible contraction points with respect to a fixed edge $e = \{ a, b \}$) \\} \label{contractpointspaceexpress}
		Fix an edge $e = \{ a, b \} \in \Delta$. Suppose that $|\St_\Delta(a)| \cup |\St_\Delta(b)|$ is convex  and consider vertices $w$ on the same side of a hyperplane $\Span(\tau)$ for a wall $\tau \in \lk_\Delta(a)$ (respectively $\eta \in \lk_\Delta(b)$) as $a$ (respectively $b$). Note that the former assumption is necessary in order for the linear contraction point space with respect to $e$ to be nonempty. \\ 
		
		Among such points, the linear contraction point space with respect to $e = \{ a, b \}$ is given by the points of the line segment $\overline{ab}$ in (the interior of) $|\St_\Delta(a)| \cup |\St_\Delta(b)|$ lying in

		\begin{align*}
			\left( \bigcap_{ \substack{ r \in \lk_\Delta(a) \cup \lk_\Delta(b) \\ \text{Facets $F \in \Delta$ such that $F \ni r$ and $F \in \St_\Delta(a)$} \\ H_r^+ \text{ the half space bounded by $H$ that contains $r$} } } \Span(((F \setminus a) \cup a'_F) \setminus r)_r^+ \right) \\ 
			\bigcap \left( \bigcap_{ \substack{ r \in \lk_\Delta(a) \cup \lk_\Delta(b) \\ \text{Facets $F \in \Delta$ such that $F \ni r$ and $F \in \St_\Delta(b)$} \\ H_r^+ \text{ the half space bounded by $H$ that contains $r$} } } \Span(((F \setminus b) \cup b'_F) \setminus r)_r^+ \right) 
		\end{align*}
		
		
		This is the largest subset $A \subset |\St_\Delta(a)| \cup |\St_\Delta(b)|$ such that the realization of $A \cup F$ in $k^d$ is convex for any facet $F \in \Delta$ formed by replacing $a$ or $b$ by $a' \notin \lk_\Delta(a) \cup \lk_\Delta(b)$ or $b' \in \lk_\Delta(a) \cup \lk_\Delta(b)$ (i.e. one wall crossing away from $|\St_\Delta(a)| \cup |\St_\Delta(b)|$) (see Lemma 1.4.1 on p. 16 of \cite{Sch}). \\

	\end{thm}
	
	\begin{cor}
		If the contraction point space of $e = \{ a, b \} \in \Delta$ is nonempty, it is a convex set whose union with a designated collection of convex sets is convex. \\
	\end{cor}

	\begin{rem} \textbf{(Duality properties) \\}
		Since $\Delta$ is flag if it has convex wall crossings (Proposition 5.3 on p. 279 of \cite{LR}), we have that $\lk_\Delta(e) = \lk_\Delta(a) \cap \lk_\Delta(b)$. Apart from being a common application of the flagness property, the condition $\lk_\Delta(e) = \lk_\Delta(a) \cap \lk_\Delta(b)$ is also the property needed for contraction of $e = \{ a, b \}$ to preserve the PL homeomorphism type by \cite{Nev}. On the edge subdivision side, recall that $\lk_{\Delta'}(v) = \Susp_{a, b}(\lk_\Delta(e))$. Since $\lk_{\widetilde{\Delta}}(w) = \lk_\Delta(a) \cup \lk_\Delta(b)$ and Theorem \ref{contractpointspaceexpress} expresses convex contractibility as unions of designated convex sets being convex sets, there is a sort of sort of duality property analogous to the relationship between polar convex sets of intersections and unions (Theorem 1.6.3 on p. 33 and Theorem 1.6.9 on p. 35 of \cite{Sch}). \\
	\end{rem}

	We now describe the sets involved in each type of wall crossing involved. \\

	\begin{lem} \textbf{(On wall data: Convexity of wall crossings with respect to $w$) \\} \label{contractpointspace1} 
		Let $\Delta$ be a locally convex simplicial pseudomanifold and $e = \{ a, b \} \in \Delta$ be an edge of $\Delta$. \\

		Recall that $|\St_{\widetilde{\Delta}}(w)| = |\St_\Delta(a)| \cup |\St_\Delta(b)|$. If $|\St_\Delta(a)| \cup |\St_\Delta(b)|$ is convex, then any point $w$ in the interior of $|\St_\Delta(a)| \cup |\St_\Delta(b)|$ induces convex wall crossings between facets of $\widetilde{\Delta}$ containing $w$ with respect to $w$. In particular, any point $w \in \overline{ab}$ is suitable. From a wall crossing convexity perspective, the vertex $w$ and a vertex $p \in \lk_{\widetilde{\Delta}}(w)$ are on the same side with respect to a hyperplane $\Span(F \setminus p)$ from a facet $F \in \Delta$ containing $p$ that is in $\St_\Delta(a) \setminus \St_\Delta(b)$ or $\St_\Delta(b) \setminus \St_\Delta(a)$. \\
	\end{lem}
	
	\begin{proof}
		Consider a facet of $F \in \Delta$ that contains exactly one of $a$ or $b$. Then, $w$ is on the same side of the hyperplane $\Span(F \setminus a)$ or $\Span(F \setminus b)$ as the off-wall vertex of the other facet of $\Delta$ containing $F \setminus a$ or $F \setminus b$. Since we assumed that $|\St_\Delta(a)| \cup |\St_\Delta(b)|$ is convex, this implies that $w$ lies inside $|\St_\Delta(a)| \cup |\St_\Delta(b)|$. This was shown in Part 3 of Proposition \ref{convsep}. \\
	\end{proof}

	\begin{lem} \textbf{(Off-wall data: $w$ as an off-wall vertex and convexity of wall crossings with respect to $r \in \lk_\Delta(a) \cup \lk_\Delta(b)$) \\} \label{contractpointspace2} 
		Let $\Delta$ be a locally convex simplicial pseudomanifold and $e = \{ a, b \} \in \Delta$ be an edge of $\Delta$. \\
		
		The space of contraction points $w$ with respect to $e = \{ a, b \}$ that can be off-wall vertices of wall crossings that are convex with respect to a fixed on-wall vertex $r \in \lk_\Delta(a) \cup \lk_\Delta(b)$ can be constructed as follows: Fix $r \in \lk_\Delta(a) \cup \lk_\Delta(b)$ and consider facets $F\in \St_\Delta(r) \cap (\St_\Delta(a) \cup \St_\Delta(b))$. They contain $r$ and at least one of $a$ or $b$. Let $a'_F$ (respectively $b'_F$) be the other vertex of $\Delta$ adjacent to the wall $F \setminus a$ (respectively $F \setminus b$). Then, intersect the half-spaces of points on the same side of the hyperplane $\Span(F \setminus r)$ as $r$ for facets of the form $(F \setminus a) \cup a'$ or $(F \setminus b) \cup b'$. \\

		This is summarized in the intersection below:

		\begin{align*}
			\left( \bigcap_{ \substack{ r \in \lk_\Delta(a) \cup \lk_\Delta(b) \\ \text{Facets $F \in \Delta$ such that $F \ni r$ and $F \in \St_\Delta(a)$} \\ H_r^+ \text{ the half space bounded by $H$ that contains $r$} } } \Span(((F \setminus a) \cup a'_F) \setminus r)_r^+ \right) \\ 
			\bigcap \left( \bigcap_{ \substack{ r \in \lk_\Delta(a) \cup \lk_\Delta(b) \\ \text{Facets $F \in \Delta$ such that $F \ni r$ and $F \in \St_\Delta(b)$} \\ H_r^+ \text{ the half space bounded by $H$ that contains $r$} } } \Span(((F \setminus b) \cup b'_F) \setminus r)_r^+ \right). 
		\end{align*}
		
		\color{black} 
		If the set of facets satisfying the properties described in the intersection is empty for a particular $r \in \lk_\Delta(a) \cup \lk_\Delta(b)$, we will take the trivial intersection (i.e. no condition is imposed and the ``intersection'' is the intersection of the previous sets). \\
		
		Alternatively, the intersection of the half-spaces involved is the largest subset $A \subset |\St_\Delta(a)| \cup |\St_\Delta(b)|$ such that the realization of $A \cup F$ in $k^d$ is convex for any facet $F \in \Delta$ formed by replacing $a$ or $b$ by $a' \notin \lk_\Delta(a) \cup \lk_\Delta(b)$ or $b' \in \lk_\Delta(a) \cup \lk_\Delta(b)$ (i.e. one wall crossing away from $|\St_\Delta(a)| \cup |\St_\Delta(b)|$) (see Lemma 1.4.1 on p. 16 of \cite{Sch}). \\ 
		
		This intersection consists of points of $|\St_\Delta(a)| \cup |\St_\Delta(b)|$ lying in the intersection of the cones centered at $a'$ or $b'$ (vertices opposite from $a$ or $b$ via a single wall crossing) whose rays are those passing through the vertices of the wall. \\
		
	\end{lem}
	
	\begin{proof}
		The wall crossings in $\widetilde{\Delta}$ using $w$ as an off-wall vertex are formed by replacing wall $a$ or $b$ with $w$ in wall crossings of $\Delta$ where $a$ or $b$ is used as an off-wall vertex. The intersections above come from thinking about the convexity of such wall crossings with respect to on-wall vertices $r \in \lk_\Delta(a) \cup \lk_\Delta(b)$. Without loss of generality, suppose that $r \in \lk_\Delta(a)$. Consider a facet $F \in \St_\Delta(a) \setminus \St_\Delta(b)$ containing $r$. By Part 3 of Proposition \ref{convsep}, the convexity with respect to $r$ of the wall crossing to $(F \setminus a) \cup a'_F$ is expressed as $a'_F$ being on the same side of $\Span(F \setminus r)$ as $r$. \\
	\end{proof}

	\begin{lem} \textbf{($w$ as an on-wall vertex and convexity with respect to $r \in \lk_\Delta(a) \cup \lk_\Delta(b)$) \\} \label{contractpointspace3}
		Let $\Delta$ be a locally convex simplicial pseudomanifold and $e = \{ a, b \} \in \Delta$ be an edge of $\Delta$. \\
		
		Suppose that $\Delta$ has convex wall crossings (which implies that $\Delta$ is flag) and the vertex $w \in \Delta$ lies on the same side as a fixed vertex $r \in \lk_\Delta(a) \setminus \lk_\Delta(b)$ with respect to the hyperplanes spanned by facets of $\lk_\Delta(r)$. Assume that $|\St_\Delta(a)|$, $|\St_\Delta(b)|$, and $|\St_{\widetilde{\Delta}}(w)|$ are convex. Then, the space of contraction points $w$ compatible with convexity of wall crossings with respect to vertices $r \in \lk_\Delta(a) \cup \lk_\Delta(b)$ adjacent to $w$ that use $w$ as another on-wall vertex are given by the following intersection: \[ \bigcap_{ \substack{m \in \lk_\Delta(a) \cap \lk_\Delta(r) \\ m \notin \Span(F \setminus  \{ a, r \} ) \\ H_r^+ \text{ contains } r } } \Span( F \setminus \{ a, r \},  m)_r^- = \bigcap_{ \substack{m \in \lk_\Delta(a) \cap \lk_\Delta(r) \\ m \notin \Span(F \setminus  \{ a, r \} ) \\ H_a^+ \text{ contains } a } } \Span( F \setminus \{ a, r \},  m)_a^+   \] 
		
		If $m \in \Span(F \setminus \{ a, r \})$ for all $m \in \lk_\Delta(a) \cap \lk_\Delta(r)$, then there are no additional conditions beyond those implied by convexity of $|\St_\Delta(a)|$ and $|\St_\Delta(w)|$. \\
		
		For $r \in \lk_\Delta(b) \setminus \lk_\Delta(a)$, the same statement holds with $a$ and $b$ switched. \\ 
		
		Finally, consider the case where $r \in \lk_\Delta(a) \cap \lk_\Delta(b)$. Note that a wall crossing using $r$ and $a$ or $b$ as an on-wall vertex is a wall crossing of facets in $\St_\Delta(a) \cup \St_\Delta(b)$ since the wall contains at least one of $a$ or $b$. Then, any linear interpolation $w$ of $a$ and $b$ (i.e. $x_w = (1 - \eta) x_a + \eta x_b$ in $k^d$ with $\eta \in (0, 1)$) preserves convexity of the support of the star around $r$ (i.e. $|\St_{\widetilde{\Delta}}(r)|$ still convex). \\

	\end{lem}
	
	\begin{proof}
		Apart from the individual wall crossings with $w$ as an off-wall vertex, we need to consider compatibility between convexity of different types of wall crossings. This includes instances where $w$ itself is an on-wall vertex and conditions required for this to be compatible with convexity with respect to other vertices of the wall. In particular, we are considering convexity of wall crossings with respect to $r \in \lk_\Delta(a)$ (respectively $r \in \lk_\Delta(b)$) involving facets $F \in \Delta$ containing $a$ (respectively $b$). Note that $a$ (respectively $b$) lies in the half-spaces in the statement. Assume without loss of generality that $r \in \lk_\Delta(a)$.  Consider facets $F \in \Delta$ such that $F \supset \{ r, a \}$. By Proposition \ref{convsep}, this is equivalent to the statement that $\Span(F \setminus \{ a, r \}, w)$ does \emph{not} separate the vertices $m$ and $r$ for any $m \in \lk_\Delta(a) \cap \lk_\Delta(r)$.  \\

		Suppose that $r \notin \lk_\Delta(a) \cap \lk_\Delta(b)$. Since $F \supset \{ r, a \}$, we have $r \in \lk_\Delta(a) \setminus \lk_\Delta(b)$. The proof for $r \in \lk_\Delta(b) \setminus \lk_\Delta(a)$ is the same with $a$ and $b$ switched. Recall that $w$ is chosen to lie on the same side as $r$ with respect to the hyperplanes spanned by facets of $\lk_\Delta(r)$ (which are walls of $\Delta$ opposite to $r$) that do \emph{not} contain $a$. In order for a hyperplane of the form $\Span(F \setminus \{ r, a \}, w)$ to separate $m$ and $r$, the intersection point with the line segment $\overline{mr}$ must come from the sum of an element of $\Span(F \setminus \{ r, a \})$ and a \emph{positive} multiple of $w$. This is because $\Span(F \setminus \{ r, a \})$ does \emph{not} intersect $\overline{mr}$ (faces intersect on faces) and the intersection point would lie inside $|\St_\Delta(r)|$. The latter condition forces a nonnegative multiple of $w$ to be involved since we would otherwise have a point that is on the opposite side of $r$ with respect to hyperplanes spanned by facets of $\lk_\Delta(r)$ that contain $F \setminus \{ r, a \}$. This can be stated more explicitly in terms of the sign of the inner product with $r^*_{(F \setminus a) \cup a'}$, where the dual is taken with respect to the vertices of $(F \setminus \{ r, a \}) \cup a'$, where $a'$ is a vertex opposite to $a$ (i.e. $r$-coefficient in the expansion with respect to this basis). Suppose that \[ \alpha + c_w x_w = (1 - \beta) x_m + \beta x_r \] holds for some $\alpha \in \Span(F \setminus \{ a, r \})$ and $\beta \in (0, 1)$. The convexity of $|\St_{\widetilde{\Delta}}(r)|$ implies that $\langle r^*_{(F \setminus a) \cup a'}, w \rangle \ge 0$ and $\langle r^*_{(F \setminus a) \cup a'}, m \rangle \ge 0$. In addition, having $\beta \ne 0$ implies that $c_w \ne 0$ since the right hand side is \emph{not} contained in $\Span(F \setminus  \{ a, r \})$.
		Since $\langle r^*_{(F \setminus a) \cup a'}, \alpha \rangle = 0$, $1 - \beta > 0$, $\beta > 0$, and $\langle r^*_{(F \setminus a) \cup a'}, r \rangle = 1$, we are forced to have $c_w > 0$. \\

		The coefficient of $r$ in the basis expansion of $w$ from vertices of $(F \setminus a) \cup a' \ni r$ is nonnegative under the initial assumptions on $w$ applied to the hyperplane $\Span(F \setminus \{ r, a \}, a')$. In order for an element of $\Span(F \setminus \{ a, r \}, w )$ to separate the points $m$ and $r$ (i.e. intersect the interior of the line segment $\overline{mr}$), this coefficient must be strictly positive. \\
		
		If $m \in \Span(F \setminus \{ a, r \} )$, separation of the vertices $m$ and $r$ by the linear subspace $\Span(F \setminus \{ a, r \}, w)$ reduces to an identity of the form \[ \mu + c_w x_w = \beta x_r \] for some $\mu \in \Span(F \setminus \{ a, r \} )$, $c_w > 0$, and $\beta \in (0, 1)$ after replacing $\alpha - (1 - \beta) x_m$ by $\mu$. This would mean that $x_w \in \Span(F \setminus a)$ since \[ c_w x_w = -\mu + \beta x_r. \] However, this is impossible since realizations of vertices of a face of a simplicial complex are linearly independent. Thus, the case where $m \in \lk_\Delta(a) \cap \lk_\Delta(r)$ and $m \in \Span(F  \setminus \{  a, r \} )$ does \emph{not} yield any examples where the linear subspace $\Span(F \setminus \{ a, r \}, w)$ separates the vertices $m$ and $r$. \\
		
		It remains to consider cases where $m \notin \Span( F \setminus \{ a, r \} )$. \\

		Even in the ``good'' cases where the contraction point $w$ preserves convexity of wall crossings, the initial assumptions on half-spaces containing $w$ force a positive coefficient of $m$ if $m \notin \Span(F \setminus a)$. Since $|\St_{\widetilde{\Delta}}(w)|$ is taken to be convex, we have that $m$ must lie on the same side of the hyperplane $\Span(F \setminus a)$ as $w$. In particular, the $a$-coefficient of $w$ with respect to the $F$-basis expansion must be strictly positive since realizations of vertices of a face of a simplicial complex are linearly independent and this coefficient is nonnegative by convexity of $|\St_\Delta(w)|$. The convexity of $|\St_\Delta(a)|$ implies that the same is true for $m$. If $m \notin \Span(F \setminus a)$, this implies that the $m$-coefficient of $w$ with respect to the $((F \setminus a) \cup m)$-basis expansion is strictly positive. Assuming initial convexity conditions on $|\St_\Delta(a)|$ and $|\St_\Delta(w)|$, this indicates that the question about separation of $m$ and $r$ is about the sign of the coefficient of $r$ with respect to the $((F \setminus a) \cup m)$-basis of $k^d$. \\ 
		
		We can describe the ``problematic contraction points $w$'' more concretely in terms of linear combinations of ($k^d$-coordinates of) vertices of $\Delta$. The coefficient of $w$ in the linear combination of vertices of $(F \setminus \{ a, r \}) \cup w$ equal to a point on the line segment $\overline{mr}$ must be \emph{strictly} positive since none of the walls of $\lk_\Delta(r)$ containing $F \setminus \{ a, r \}$ span hyperplanes separating $m$ and $r$. Then, the intersection point of the hyperplane $\Span(F \setminus \{ a, r \}, w)$ and the line segment $\overline{mr}$ can be expressed as \[ \alpha + c_w x_w = (1 - \beta) x_m + \beta x_r \] for some $\alpha \in \Span(F \setminus \{ a, r \})$, $c_w > 0$, and $\beta \in (0, 1)$. \\

		Rearranging terms, we have that \[ c_w x_w = -\alpha + (1 - \beta) x_m + \beta x_r. \] Since $c_w > 0$, this implies that the coefficient of $r$ in the expansion of $w$ with respect to this basis is strictly positive. We can think about this as $w$ being on the same side of the hyperplane $\Span(F \setminus \{ a, r \}, m)$ as $r$. \\

		Next, we consider the ``converse'' question of implications of such linear relations. Suppose that  \[ x_w = \delta + \xi_m x_m + \xi_r x_r \] for some $\delta \in \Span(F \setminus \{ a, r \})$, $\xi_m > 0$, and $\xi_r > 0$. Then, we have that 
		
		\begin{align*}
			x_w - \delta &= \xi_m x_m + \xi_r x_r \\
			\Longrightarrow \frac{1}{\xi_m + \xi_r} x_w - \frac{1}{\xi_m + \xi_r} \delta &= \frac{\xi_m}{\xi_m + \xi_r} x_m + \frac{\xi_r}{\xi_m + \xi_r} x_r.
		\end{align*}
		
		Since $\frac{\xi_m}{\xi_m + \xi_r}, \frac{\xi_r}{\xi_m + \xi_r} \in (0, 1)$ and $\frac{\xi_m}{\xi_m + \xi_r} + \frac{\xi_r}{\xi_m + \xi_r} = 1$, this implies that $\frac{1}{\xi_m + \xi_r} x_w - \frac{1}{\xi_m + \xi_r} \delta$ is a point on the interior of the line segment $\overline{mr}$ and the hyperplane $\Span(F \setminus \{ a, r \}, w)$ separates $m$ and $r$. \\
		
		The ``problematic'' cases above are instances where $w$ is on the same side of $\Span(F \setminus \{ a, r \}, m)$ as $r$ for some $m \in \lk_\Delta(a) \cap \lk_\Delta(r)$. Such vertices $m$ are the only vertices where $\overline{mr}$ could be separated by the hyperplane $\Span(F \setminus \{ a, r \}, m)$. This is because at most one side of any particular wall of $\Delta$ containing $r$ (e.g. $(F \setminus \{ a, r \}) \cup r = F \setminus a$) is ``open'' (i.e. not already ``occupied''). Note that at most one of the two vertices adjacent to a wall is equal to $a$ or $b$ (which are \emph{not} opposite to each other since opposite off-wall vertices of a wall crossing cannot be adjacent when $\Delta$ is flag).  \\

		Finally, we consider the case $r \in \lk_\Delta(a) \cap \lk_\Delta(b)$. Assume that $x_w = (1 - \eta) x_a + \eta x_b$ for some $\eta \in (0, 1)$. Suppose that there is a hyperplane of the form $\Span(F \setminus \{ a, r \}, w)$ or $\Span(H \setminus \{ b, r \}, w)$ separating two vertices $m$ and $r$ with $m \in \lk_{\widetilde{\Delta}}(r)$. Note that $F \setminus a \in \lk_\Delta(a)$ implies that $F \setminus a \in \lk_{\widetilde{\Delta}}(w)$. We will focus on the first case since the second case uses the same reasoning. The convexity of $|\St_{\widetilde{\Delta}}(r)|$ implies that the hyperplane $\Span(F \setminus \{ a, r \}, w)$ does \emph{not} separate $r$ and any vertex of $\lk_{\widetilde{\Delta}}(r)$. Note that the vertices of $\lk_{\widetilde{\Delta}}(r)$ consist of vertices of $\lk_\Delta(r) \setminus \{ a, b \}$ and $w$. Since $|\St_{\widetilde{\Delta}}(r)| \subset |\St_\Delta(r)|$, we have covered all the possible vertices $m \in \lk_\Delta(r)$ such that the hyperplane $\Span( F \setminus \{ a, r \}, w)$ could separate $m$ and $r$. \\

	\end{proof}
	
	\pagebreak 
	
	\begin{lem} \textbf{(Redundancy from $w$ as an on-wall vertex that is not the ``center of convexity'') \\} \label{contractpointspace4} 
		Let $\Delta$ be a locally convex simplicial pseudomanifold and $e = \{ a, b \} \in \Delta$ be an edge of $\Delta$.  \\
		
		The convexity of $|\St_\Delta(p)|$ for vertices $p \in \Delta$ implies that $b$ lies in the half-spaces from Lemma \ref{contractpointspace3} associated to $r \in \lk_\Delta(a) \setminus \lk_\Delta(b)$ and $a$ lies in the analogous half-spaces associated to $r \in \lk_\Delta(b) \setminus \lk_\Delta(a)$. In particular, this implies that any point on the line segment $\overline{ab}$ lies in the half-spaces from Lemma \ref{contractpointspace3}. \\
	\end{lem}
	
	\begin{proof}
		Without loss of generality, take $r \in \lk_\Delta(a) \setminus \lk_\Delta(b)$. We will show that $b$ also lies in the half-spaces of Lemma \ref{contractpointspace3}. The convexity of half-spaces would then imply that the entire line segment $\overline{ab}$ lies in these half-spaces. The convexity of $|\St_\Delta(a)|$ and $|\St_\Delta(r)|$ are convex implies that the intersection $|\St_\Delta(a)| \cap |\St_\Delta(r)|$ is also convex. Since $b \in \lk_\Delta(a)$ but $b \notin \lk_\Delta(r)$, this means that $b$ is on the ``wrong side'' (i.e. the opposite side of $r$) with respect to some half-space bounded by the span of a wall of $\Delta$ of the form $F \setminus r$ for some facet $F \in \Delta$ such that $F \supset \{ a, r \}$. This means that the line segment $\overline{br}$ intersects the face $F \setminus r$ at some point $p$ that lies inside some face $G \subset F \setminus r$. Consider a vertex $m \in \lk_\Delta(a) \cap \lk_\Delta(r)$. The assumption $m \notin \Span(F \setminus \{ a, r \})$ from Lemma \ref{contractpointspace3} implies that it does \emph{not} lie in $F$. \\

		Assume that $m \notin \Span(F \setminus \{ a, r \} )$ since the ``degenerate'' case $m \in \Span(F \setminus \{ a, r \} )$ is trivial by the proof of Lemma \ref{contractpointspace3}. We split into cases according to whether $a \in \Span(F \setminus \{ a, r \}, m)$ or not. \\
			
		If $a \in \Span(F \setminus \{ a, r \}, m)$, then we have that $\Span(F \setminus r, m) = \Span(F \setminus \{ a, r \}, m)$. If $p$ is in the convex hull of the vertices of $F \setminus r = (F \setminus \{ a, r \}) \cup a$, the assumption $a \in \Span(F \setminus \{ a, r \}, m)$ implies that $p \in \Span(F \setminus \{ a, r \}, m)$ since $p$ would be a convex linear combination of the vertices of $F \setminus r = (F \setminus \{ a, r \}) \cup a$ and $a$ can be replaced by an element of $\Span(F \setminus \{ a, r \}, m)$ in the linear combination. Since $\Span(F \setminus r) \subset \Span(F \setminus \{ a, r \} ), m)$ and $\Span(F \setminus r)$ is a hyperplane in $k^d$, this implies that $\Span(F \setminus r) = \Span(F \setminus \{ a, r \}, m)$ since $F \setminus r \in \Delta$ and $a \notin \Span(F \setminus \{ a, r \} )$. Then, $b$ being on the opposite side of $r$ with respect to the hyperplane $\Span(F \setminus r)$ is the same as being on the opposite side of $\Span(F \setminus \{ a, r \}, m)$. Thus, the hyperplane $\Span(F \setminus \{ a, r \}, m)$ separates the vertices $b$ and $r$ in this case and $b$ lies in the half-space from Part 3 associated to $m \in \lk_\Delta(a) \setminus \lk_\Delta(b)$. \\

		Suppose that $a \notin \Span(F \setminus \{ a, r \}, m)$. Then, the vertices of $(F \setminus r) \cup m$ form a basis of $k^d$. Note that $b$ lies on the same side of the hyperplane $\Span(F \setminus a)$ as $r$ (i.e. $b$ and $r$ \emph{not} separated by this hyperplane). Since $b \notin F \setminus a$, $r \in F \setminus a$, and $p \ne r$, we have that $p \notin F \setminus a$ (as a realization). This means that the coefficient of $a$ in the $F$-basis expansion of $p$ (a convex linear combination of vertices of $F \setminus r$) is strictly positive and $p \notin \Span(F \setminus a)$. If $r \in \Span(F \setminus \{ a, r \}, m)$, this indicates that the vertex $b$ is on the ``opposite side'' of $r$ with respect to the hyperplane $\Span(F \setminus \{ a, r \}, m)$ by default. This is because $r$ is contained in the hyperplane bounding the half-space in question. Note that $r \in \Span(F \setminus \{ a, r \}, m)$ implies that $\Span(F \setminus \{ a, r \}, m) = \Span(F \setminus a, m)$ and $m \in \Span(F \setminus a)$ since $\dim \Span(F \setminus \{ a, r \}, m) = d - 1$ in our setting. In particular, this means that $\Span(F \setminus  \{a, r\}, m) = \Span(F \setminus a)$. Both $a$ and $b$ are on the same side of this hyperplane. \\

		Assume that $r \notin \Span(F \setminus \{ a, r \}, m)$. Then, we have a setting that is similar to a wall crossing except that the simplicies associated to $(F \setminus r) \cup m$ and $(F \setminus a) \cup m$ do \emph{not} form faces of $\Delta$. If $p \in \Span(F \setminus \{ a, r \} )$, then $p \in \Span(F \setminus \{ a, r \}, m)$ and the hyperplane $\Span(F \setminus \{ a, r \}, m)$ already separates $b$ and $r$. Next, we consider the ``general'' case $p \notin \Span(F \setminus  \{ a, r \} )$. Since $a \notin \Span(F \setminus \{ a, r \}, m)$ and the coefficient of $a$ in the expansion $p = \alpha + ca$ as a convex linear combination of vertices of $F \setminus r = F \setminus \{ a, r \} \cup r$ is strictly positive (i.e. $c > 0$), we have that $p \notin \Span(F \setminus \{ a, r \}, m)$. Otherwise, we would have $a \in \Span(F \setminus \{ a, r \}, m)$. Note that $m$ lies on the same side of the hyperplane $\Span(F \setminus a)$ as $a$ since $|\St_\Delta(a)|$ is convex. In other words, the $a$-coefficient of $m$ with respect to the $F$-basis expansion of $k^d$ is strictly positive. The nonnegativity follows from convexity. If $m \in \Span(F \setminus a) = \Span(F \setminus \{ a, r \}, r)$, then we would have $r \in \Span(F \setminus \{ a, r \}, m)$ since we assumed that $m \notin \Span(F \setminus  \{ a, r \})$. However, we are working in the case where $r \notin \Span(F \setminus \{ a, r \}, m)$. \\

		Similarly, the convexity of $|\St_\Delta(r)|$ implies that the coefficient of $r$ in the $F$-basis expansion of $m$ is nonnegative. This is strictly positive since we are in the case where the vertices of $(F \setminus r ) \cup m$ form a basis of $k^d$. Otherwise, we would have $m \in \Span(F \setminus r)$ and the vertices of $(F \setminus r) \cup m$ do \emph{not} form a basis of $k^d$. After ``solving'' for $r$ in the $F$-basis expansion of $m$, we can see that the coefficient of $a$ in the $((F \setminus r) \cup m)$-basis expansion of $r$ is strictly negative since we are assuming that $r \notin \Span(F \setminus \{ a, r \}, m)$. If $x_m = \beta + b x_r$ with $b > 0$ is the $F$-basis expansion, we have that $b x_r = x_m - \beta$ with $b > 0$ in the $((F \setminus r) \cup m)$-basis expansion. Recall that $p$ is a convex linear combination of the vertices of $F \setminus r$ with a strictly positive coefficient of $a$ and $p \notin \Span(F \setminus \{ a, r \}, m)$. Since the $a$-coefficient of $p$ with respect to the $((F \setminus r) \cup m)$-basis expansion is strictly positive, this implies that there is a point $z$ on the line segment $\overline{pr}$ with $a$-coefficient with respect to the $((F \setminus r) \cup m)$-basis equal to $0$. This would mean that $z \in \Span(F \setminus \{ a, r \}, m)$ and the points $p$ and $r$ (and thus $b$ and $r$) are separated by the hyperplane $\Span(F \setminus \{ a, r \}, m)$.

	\end{proof}

	\begin{rem} \textbf{(Connections between on-wall convexity properties and other criteria) \\}
		The criteria in Lemma \ref{contractpointspace2} and Lemma \ref{contractpointspace3} can be used to characterize individual vertices or pairs of adjacent vertices preserving convexity of the union of stars over $a$ and $b$. This is relevant to checking (non)emptiness of linear contraction point spaces in many examples.  \\
	\end{rem}

	\color{black}

	\subsection{Restrictions on possible realizations and convexity-preserving spaces} \label{convpresrest}
	
	As an initial application of Theorem \ref{contractpointspaceexpress}, we note that even the simple union in Lemma \ref{contractpointspace1} is a useful tool for ruling out convex contractibility of an edge. An example we consider is from boundaries of cross polytopes. In particular, the on-wall condition Lemma \ref{contractpointspace1} causes the failure of convex contractibility although the more ``complicated'' off-wall condition of Lemma \ref{contractpointspace2} is satisfied. The same reasoning will later be used to describe shrinking contraction point spaces in certain cases where they may be expected to increase (Theorem \ref{exsubdivcon}). For applications of this space in other parts of this section, we will focus on more general geometric properties of convex sets and simplicial pseudomanifolds before thinking about connections to PL homeomorphisms (e.g. results leading up to Theorem \ref{convedconst}). \\

	\begin{exmp} \textbf{(Convexity-preserving edge contractions and cross polytopes) \\} \label{contrcrosspolytop}
		We note that boundaries of cross polytopes are minimal with respect to edge subdivisions/contractions among flag simplicial spheres PL homeomorphic to it since every edge is contained in an induced 4-cycle (Theorem 1.2 on p. 70 and p. 77 of \cite{LN}). In other words, contracting an edge would give rise to a simplicial complex that is not flag. This also means that edge contractions cannot preserve convexity of wall crossings since the former must imply flagness (Proposition 5.3 on p. 279 of \cite{LR}). \\
		
		The nonexistence of an edge contraction preserving convexity of wall crossings is consistent with the failure to satisfy the conditions of Lemma \ref{contractpointspace1}. Without loss of generality, suppose that the edge we are considering is $\{ 1, 2 \} \in \Delta$. Then, $\lk_\Delta(1)$ is the cross polytope formed by the vertices $\pm 2, \pm 3, \ldots, \pm d$ and $\lk_\Delta(2)$ is the cross polytope formed by the vertices $\pm 1, \pm 3, \pm 4, \ldots \pm d$. Note that $\St_\Delta(1) = \Cone_1(\lk_\Delta(1))$ and $\St_\Delta(2) = \Cone_2(\lk_\Delta(2))$ in our setting. Although $-1 \in \St_\Delta(2)$ and $-2 \in \St_\Delta(1)$, the edge $\{ -1, -2 \} \in \Delta$ does \emph{not} lie in $|\St_\Delta(1)| \cup |\St_\Delta(2)|$. This implies that the space $|\St_\Delta(1)| \cup |\St_\Delta(2)|$ is \emph{not} convex and the contraction of the edge $\{ 1, 2 \} \in \Delta$ does \emph{not} preserve convexity of wall crossings/stars over vertices. Since the same argument can be repeated for \emph{any} edge $\{ a, b \} \in \Delta$ with $a$ and $b$ replacing $1$ and $2$, the cross polytope $\Delta$ does \emph{not} have any edges whose contraction preserves convexity of wall crossings/stars over vertices. \\

		However, the conditions involving the compatibility with conditions for contraction points to be off-wall vertices opposite to $a$ or $b$ compatible with convex wall crossings (Lemma \ref{contractpointspace2}) are still satisfied. The wall crossings all consist of replacing one of the vertices $p$ by its antipodal vertex $-p$. For any vertex $p$ of the cross polytope, the rays from this vertex to the vertices that it is adjacent to forms a cone containing the entire cross polytope. This already includes every vertex except $-p$ by definition. By convexity, it suffices to show that $-p$ lies in this cone. Note that $\sum_{q \in \Delta \setminus p} (e_q - e_p) = -2(d - 1) e_p$, which is a positive multiple of $-e_p$. After dividing by $2(d - 1)$, we find that $-e_p$ lies in the cone. This implies that the half-space intersection condition is trivial on $\lk_\Delta(a) \cup \lk_\Delta(b)$ (which is equal to $\Delta$ for any $\{ a, b \} \in \Delta$ when $\Delta$ is a cross polytope). Then, the only possibly ``interesting'' condition is from Lemma \ref{contractpointspace3}. Since $m \in \lk_\Delta(a) \cap \lk_\Delta(r)$ ($r = -2 \in \lk_\Delta(1) \setminus \lk_\Delta(2)$), we have that $\Span(F \setminus \{ a, r \} ) = \Span(3, \ldots, d)$ and the half-space condition from Lemma \ref{contractpointspace3} is trivial. \\

	\end{exmp}

	The expression of the contraction point space from Theorem \ref{contractpointspaceexpress} also implies that nearby convex contractible edges can force affine/linear restrictions on realizations of facets containing these vertices (Proposition \ref{starunintspan}). In particular, this forces the flatness of certain wall crossings from walls with a common span (Proposition \ref{convwallspan}) and extreme behavior of linear contraction point spaces after certain external external edge contractions (Theorem \ref{convedconst}). While convex contractibility of edges requires a kind of ``additional convexity'', it is interesting to note that simultaneous convex contractibility can induce flatness of nearby wall crossings. In general, this can be viewed from the perspective of ``snapshots'' of changes in the contraction point space induced by PL homeomorphisms involving external edge subdivisions/contractions. \\

	Apart from Theorem \ref{contractpointspaceexpress} and convexity properties interpreted as (non)separation by hyperplanes, we also use some arguments involving connectivity properties of simplicial pseudomanifolds. In this context, we discuss local topological changes induced by edge contractions. \\
	
	\begin{exmp} \textbf{(Edge contractions and topology of components involved) \\}
		When we consider the contraction of an edge $e = \{ a, b \} \in \Delta$ of a $(d - 1)$-dimensional flag simplicial pseudomanifold $\Delta$ (e.g. one that is locally convex by Proposition 5.3 on p. 279 -- 280 of \cite{LR}), the affected faces are those containing $a$ or $b$. Note that $\lk_\Delta(a)$ and $\lk_\Delta(b)$ are $(d - 2)$-dimensional flag spheres (see Definition \ref{pseudomfldwallcross} and p. 100 of \cite{St}). Then, $\St_\Delta(a) = \lk_\Delta(a) * a$ and $\St_\Delta(b) = \lk_\Delta(b) * b$ are PL-homeomorphic to $(d - 1)$-dimensional closed disks. Since $\Delta$ is flag and $e = \{ a, b \} \in \Delta$, we have that $\lk_\Delta(e) = \lk_\Delta(a) \cap \lk_\Delta(b)$. Since $\lk_\Delta(e)$ is a $(d - 3)$-dimensional (flag) simplicial sphere, we have that $\St_\Delta(e) = \lk_\Delta(e) * a * b$ is PL-homeomorphic to a $(d - 1)$-dimensional closed disk. \\
		
		Putting this together, we have that $|\St_\Delta(a)|$ and $|\St_\Delta(b)|$ are $(d - 1)$-dimensional closed disks whose intersection is a smaller $(d - 1)$-dimensional closed disk $|\St_\Delta(e)|$. When we contract the edge $e = \{ a, b \}$ to a point $w$, this intersection is contracted to a $(d - 2)$-dimensional closed disk $\St_{\widetilde{\Delta}}(w)$ in the contraction $\widetilde{\Delta}$ of $\Delta$ with respect to the edge $e = \{ a, b \} \in \Delta$. If we move from $\widetilde{\Delta}$ to $\Delta$, we are ``opening up'' a $(d - 2)$-dimensional closed disk in a way that doubles the walls of $\widetilde{\Delta}$ containing $w$ (use $a$ and $b$ in $\Delta$ instead). Note that $\Cone_w(S)$ forms a $(d - 2)$-dimensional closed disk that forms a ``diameter'' splitting the $(d - 1)$-dimensional closed disk $|\St_{\widetilde{\Delta}}(w)$ into two parts. \\
	\end{exmp}

	The affine/linear restrictions are related to ``flatness'' of intersections of boundaries of convex sets. This is expressed in terms of invariance of spans of walls on intersections involved in our setting below. \\
	
	\begin{prop} \textbf{(Intersection of convex star unions) \\} \label{starunintspan}
		Suppose that $e = \{ a, b \} \in \Delta$ is an edge such that $|\St_\Delta(a)| \cup |\St_\Delta(b)|$ is convex. Take vertices $y$ and $z$ are adjacent vertices of $\Delta$ that are opposite to $a$ or $b$ via a single wall crossing using a wall of $\lk_\Delta(a) \cup \lk_\Delta(b)$. If $|\St_\Delta(y)| \cup |\St_\Delta(z)|$ is convex, then the walls of $\Delta$ in $(\lk_\Delta(a) \cup \lk_\Delta(b)) \cap (\lk_\Delta(y) \cup \lk_\Delta(z))$ all have the same span. \\
	\end{prop}
	
	\begin{proof}
		Fix a wall $\tau \in (\lk_\Delta(a) \cup \lk_\Delta(b)) \cap (\lk_\Delta(y) \cup \lk_\Delta(z))$. Without loss of generality, suppose that $\tau \in \lk_\Delta(a) \cap \lk_\Delta(y)$. This is a wall that is in the boundaries $\partial(|\St_\Delta(a)| \cup |\St_\Delta(b)|)$ and $\partial( |\St_\Delta(y)| \cup |\St_\Delta(z)|)$. Since $|\St_\Delta(a)| \cup |\St_\Delta(b)|$ is convex, all points of $|\St_\Delta(a)| \cup |\St_\Delta(b)|$ lie on the same side of the hyperplane $\Span(\tau)$ as $a$ (Corollary 1.3.5 on p. 12 of \cite{Sch}). Similarly, the convexity of $|\St_\Delta(y)| \cup |\St_\Delta(z)|$ implies that all points of $|\St_\Delta(y)| \cup |\St_\Delta(z)|$ lie on the same side of $\Span(\tau)$ as $y$. Putting these together, the points of the intersection $ (\lk_\Delta(a) \cup \lk_\Delta(b)) \cap (\lk_\Delta(y) \cup \lk_\Delta(z))$ must lie \emph{on} the hyperplane $\Span(\tau)$. For example, this includes the vertices of any wall of $\Delta$ lying in $(\lk_\Delta(a) \cup \lk_\Delta(b)) \cap (\lk_\Delta(y) \cup \lk_\Delta(z))$ (e.g. one opposite to $b$ and $z$). Dimension reasons then imply that the walls of $\Delta$ in $(\lk_\Delta(a) \cup \lk_\Delta(b)) \cap (\lk_\Delta(y) \cup \lk_\Delta(z))$ all have the same span. \\
	\end{proof}
	
	The result above gives initial constraints on nearby convex contractible edges. \\
	
	\begin{cor} \textbf{(Non-repeated wall spans and preservation of contraction spaces with respect to $e = \{ a, b \}$) \\}
		\begin{enumerate}
			\item Suppose that $|\St_\Delta(a)| \cup |\St_\Delta(b)|$ is convex and that there is an edge $\{ y, z \} \in \Delta$ where $y$ \textcolor{blue}{or} $z$ are opposite to $a$ or $b$. If there is a pair of walls in $\lk_\Delta(a) \cup \lk_\Delta(b)$ opposite to $y$ and $z$ with different spans, then the contraction of $\{ y, z \}$ does \emph{not} preserve convexity. \\
			
			\item Suppose that the assumption in Part 1 holds for \emph{any} pair of adjacent vertices $y$ and $z$ opposite to $a$ and $b$. Then, any convexity-preserving contractions of edges that do \emph{not} lie in $|\St_\Delta(a)| \cup |\St_\Delta(b)|$ must preserve the contraction point space with respect to $e = \{ a, b \}$. \\
		\end{enumerate}
	\end{cor}

	The first connectivity property of this ``flat'' region with a common span is a consequence of convexity properties implied by Lemma \ref{contractpointspace1}. \\
	
	\color{black} 
	\begin{prop} \textbf{(Adjacency constraints within unions of stars) \\}
		\label{starunionadj}
		Let $\Delta$ be a simplicial pseudomanifold and $e = \{ a, b \} \in \Delta$ be an edge of $\Delta$. Take vertices $r \in \lk_\Delta(a) \setminus \lk_\Delta(b)$ and $s \in \lk_\Delta(b) \setminus \lk_\Delta(a)$. \\
		\begin{enumerate}
			\item If $\{ r, s \} \in \Delta$, then the intersection of the line segment $\overline{rs}$ with $|\St_\Delta(a)| \cup |\St_\Delta(b)|$ consists of the disjoint vertices $r$ and $s$ and does \emph{not} include any point of the interior of $\overline{rs}$. \\
			
			\item If $|\St_\Delta(a)| \cup |\St_\Delta(b)|$ is convex, then $\{ r, s \} \notin \Delta$. \\
		\end{enumerate}
	\end{prop}
	
	\begin{proof}
		\begin{enumerate}
			\item Since $s \notin \lk_\Delta(a)$, we have that $\{ r, s \} \notin \lk_\Delta(a)$. Similarly, $\{ r, s \} \notin \lk_\Delta(b)$ since $r \notin \lk_\Delta(b)$. This implies that $\{ r, s \} \notin \St_\Delta(a) \cup \St_\Delta(b)$ and the line segment $\overline{rs}$ crosses at least one wall of $\Delta$.  \\
			
			Recall that the intersection of two faces of a simplicial complex is a face of the starting faces. If these are distinct faces, the intersection lies on the boundary of each of them. Applying this to the setting above, we have that the intersection of $\{ r, s \}$ with a face of $\St_\Delta(a) \cup \St_\Delta(b)$ is at most a singleton (vertex $r$ or $s$). This yields the desired statement since the support is the union of realization of faces in the given simplicial complex. \\
			
			\item Since $r \in \lk_\Delta(a)$ and $s \in \lk_\Delta(b)$, we have that $r, s \in |\St_\Delta(a)| \cup |\St_\Delta(b)|$. The convexity of $|\St_\Delta(a)| \cup |\St_\Delta(b)|$ would then imply that the entire line segment $\overline{rs}$ is contained in $|\St_\Delta(a)| \cup |\St_\Delta(b)|$. Suppose that $\{ r, s \} \in \Delta$. Then, Part 1 implies that intersection of any particular face of $\St_\Delta(a) \cup \St_\Delta(b)$ is empty or a single vertex given by $r$ or $s$ by Part 1 and $\overline{rs} \cap |\St_\Delta(a)| \cup |\St_\Delta(b)| \subset \{ r, s \}$. However, this would contradict our assumption that $|\St_\Delta(a)| \cup |\St_\Delta(b)|$ is convex. Thus, we have that $\{ r, s \} \notin \Delta$ if $|\St_\Delta(a)| \cup |\St_\Delta(b)|$ is convex. 
		\end{enumerate}
	\end{proof}

	Another key input is that walls in the intersection with the same span are connected via wall crossings. In order to show this, we combine previous results with strong connectivity properties of simplicial pseudomanifolds and properties of realizations of simplicial complexes (e.g. distinct faces intersecting on the boundary). \\
	
	\color{black} 
	\begin{prop} \label{convwallspan}
		Let $\Delta$ be a $(d - 1)$-dimensional simplicial pseudomanifold. Suppose that $|\St_\Delta(r)| \subset k^d$ is a $(d - 1)$-dimensional polytope and $\Delta$ has convex wall crossings with respect to $r$. Consider a wall $\tau \in \lk_\Delta(r)$. The collection of walls of $\lk_\Delta(r)$ with span equal to $\Span(\tau)$ form a simplicial pseudomanifold. In particular, this means that such facets are connected to each other via wall crossings inside $\lk_\Delta(r)$. As wall crossings in $\St_\Delta(r)$ with $r$ on the wall, they are flat with respect to $r$. \\
		
	\end{prop}

	\begin{rem}
		Some connectivity observations may be generalized to $(d - 1)$-dimensional simplicial pseudomanifolds $S$ that have realizations with support $|S|$ given by $(d - 1)$-dimensional convex polytopes. Note that $\Span(\tau) \cap |\St_S(p)|$ is a $(d - 1)$-dimensional closed disk either containing $p$ in its or its boundary (see Corollary 1.16 on p. 24, definitions on p. 26, and Lemma 1.17 on p. 27 of \cite{Hud}). Another key input would be the fact that the intersection of the $k^d$ realization of a face $G$ of $S$ with a non-separating hyperplane hyperplane is a (possibly empty) face of $G$. \\
	\end{rem}

	\begin{proof}
		Note that a finite intersection of convex sets is convex and that convex sets are connected. We would like to use the initial conditions of the problem to make a stronger statement about how the components are connected. \\
		
		Since the support $|\St_\Delta(r)|$ of $\St_\Delta(r)$ (i.e. union of facets of $\St_\Delta(r)$ in $k^d$) is \emph{not} an affine subspace or a half-space inside the ambient space, $|\St_\Delta(r)|$ is the convex hull of its boundary points. Recall that every nonempty convex set that is not an affine space or a half-space is equal to the intersection of its supporting half-spaces (Corollary 1.3.5 on p. 12 of \cite{Sch}). In our case, we have that $\Span(\tau)$ forms the boundary of a supporting half-space of $|S|$. Note that the intersection of a set $A$ and a union of sets is the union of intersections of each of these sets with $A$. \\

		Let $S = \St_\Delta(r)$. Start with a codimension 1 face $A$ of a wall of $S$ (i.e. codimension 2 in S) with span equal to $\Span(\tau)$. Consider convex linear combinations of the vertices of $A$ with a vertex $m \in |S| \cap \Span(\tau)$ \emph{not} in A. In order to cover parts ``in between'' $A$ and $m$, convex linear combinations of vertices of $A$ and $m$ must also be contained in a wall of $S$ with span equal to $\Span(\tau)$. This is because $|S| \cap \Span(\tau)$ is the union of intersections of facets of $S$ with $\Span(\tau)$. To be more precise, each new vertex introduced via a wall crossing via a linear combination of $A$ and $m$ lies in $\Span(S)$. We can see this by considering one wall crossing at a time. For example, a single wall crossing would mean points in the interior of the convex linear combination of a single new vertex \emph{outside} $\Span(\tau)$ and vertices of $A$ are \emph{not} in $\Span(\tau)$. We can then repeat this argument for subsequent wall crossings after $A$ has a single vertex replaced by a new one. Let $L$ be the span of a wall of $\lk_\Delta(r)$ that is crossed by convex linear combinations of $m$ and $A$. The argument above boils down to a strict containment $L \supset \Span(G)$ of a codimension 2 face $G \in \Delta$ such that $\Span(G) \subset \Span(\tau)$ and $L \subset \Span(\tau)$. This would imply that $L = \Span(\tau)$. Note that $|S| \cap \Span(\tau)$ is convex since the intersection of a pair of convex sets is convex. \\

		Since the span of the vertices excluding $r$ is equal to $\Span(\tau)$ for each of the facets in the wall crossings described above, the vertex $r$ is \emph{not} involved in the associated wall relations. This implies that the wall crossings in $\Delta$ are flat with respect to $r$.

	\end{proof}

	The connectivity and affine/linear ``flatness'' properties yield constraints on nearby convex contractible edges. We first obtain restrictions on possible wall crossings and the number of walls in between a given pair of ``opposite vertices'' of the pair of contractible edges considered. Returning to convex contractibility, the earlier restrictions imply that linear contraction point spaces of nearby convex contractible edges can have extreme behavior. \\
	
	\color{black} 
	
	\begin{thm} \textbf{(Convex contractible edge adjacency constraints) \\}
		\label{convedconst}
		Suppose that a $(d - 1)$-dimensional simplicial pseudomanifold $\Delta$ has convex wall crossings and $e = \{ a, b \} \in \Delta$ is an edge of $\Delta$ such that $|\St_\Delta(a)| \cup |\St_\Delta(b)|$ is convex. Let that $y$ and $z$ be adjacent vertices of $\Delta$ opposite from $a$ or $b$ via a single wall crossing using $a$ or $b$ as off-wall vertices. \\ 
		
		\begin{enumerate}
			\item Suppose that $|\St_\Delta(y)| \cup |\St_\Delta(z)|$ is convex. Consider wall crossings among facets of $(\St_\Delta(a) \cup \St_\Delta(b)) \cap (\St_\Delta(y) \cup \St_\Delta(z))$. Given a fixed on-wall vertex from $r \in \{ a, b \}$ and another one chosen from $s \in \{ y, z \}$, there are at most $d$ walls of $\Delta$ in $\lk_\Delta(r) \cap \lk_\Delta(s)$. The material in Proposition \ref{convwallspan} points towards possible generalizations following the contraction of the edges $\{ a, b \}$ and $\{ y, z \}$ if they preserve convexity of wall crossings. These walls are connected to each other via wall crossings in $\Delta$ keeping $r$ or $s$ on the wall. \\  
			
			\item Suppose $\{ y, z \} \in \Delta$ is convex contractible (Definition \ref{contptspaceconvcon}) to a point in the \emph{interior} of the line segment $\overline{yz}$. Then, contracting the edge $\{ y, z \}$ yields a contraction point space with respect to the edge $\{ a, b \}$ that is either at most one point or the entire line segment $\overline{ab}$. \\

		\end{enumerate}
	\end{thm}
	
	\begin{proof}
		\begin{enumerate}
			\item Without loss of generality, consider a wall $\tau$ of $\Delta$ lying in $(\lk_\Delta(a) \cup \lk_\Delta(b)) \cap (\lk_\Delta(y) \cup \lk_\Delta(z))$ that is opposite to $a$ and $y$. If there is another wall opposite to both $a$ and $y$, then Proposition \ref{starunintspan} and Proposition \ref{convwallspan} imply that there is one obtained by a ``wall crossing'' replacing a single vertex of $\tau$. Let $\omega \in \Delta$ be such a wall. Then, $\omega = (\tau \setminus p) \cup q$ for some $p \in \tau$ and $q \in (\lk_\Delta(a) \cup \lk_\Delta(b)) \cap (\lk_\Delta(y) \cup \lk_\Delta(z))$. Note that $\Span(\omega) = \Span(\tau)$ by Proposition \ref{starunintspan}. \\
			
			We will consider convexity of wall crossings with respect to vertices of $\tau \cap \omega = \tau \setminus p$ and $q = \omega \setminus (\tau \cap \omega)$. Let $F \coloneq \tau \cup a$, $F' \coloneq \tau \cup y$, $G \coloneq \omega \cup a = (F \setminus p) \cup q$, and $G' \coloneq \omega \cup y = (F' \setminus p) \cup q$. We claim that the wall crossings associated to $F \cup F'$ and $G \cup G'$ are flat with respect to $p$ and $q$ respectively. Since the wall crossing associated to $G \cup G'$ is convex with respect to $q$, the hyperplane $\Span(\omega \setminus q, a) = \Span(\tau \setminus p, a)$ does \emph{not} separate the vertices $q$ and $y$. Similarly, the convexity of the wall crossing associated to $F \cup F'$ with respect to $p$ implies that the hyperplane $\Span(\tau \setminus p, a)$ does \emph{not} separate $p$ and $y$. Since $G = (F \setminus p) \cup q$, the facet $G$ is obtained from $F$ via a wall crossing with $p$ and $q$ as off-wall vertices. This implies that $p$ and $q$ lie on opposite sides of the hyperplane $\Span(\tau \setminus p, a)$. Thus, we have that $y \in \Span(\tau \setminus p, a) = \Span(\omega \setminus q, a)$. In particular, this means that the wall crossing associated to $F \cup F'$ is flat with respect to $p$ and that the one associated to $G \cup G'$ is flat with respect to $q$. \\

			We can repeat similar arguments for subsequent wall crossings within $(\lk_\Delta(a) \cup \lk_\Delta(b)) \cap (\lk_\Delta(y) \cup \lk_\Delta(z))$. For example, suppose that $\eta$ is a wall of $\Delta$ formed by replacing a vertex $m \in \omega$ by a vertex $r$ in a wall crossing within $\lk_\Delta(a) \cap \lk_\Delta(y)$. Then, the the vertex $y$ lies in a codimension 2 linear subspace. We can generalize this argument to higher dimensions. In each turn, the new vertex we are flat with respect to is \emph{not} in the span of the previous ``centers of convexity'' considered. If there are $d$ consecutive ``wall crossings'' of $\lk_\Delta(a) \cap \lk_\Delta(y)$ (e.g. if there are $d + 1$ such facets of $\Delta$), the steps above would each decrease the dimension of linear subspaces we can choose $y$ from by 1. In other words, each wall crossing among facets of $\lk_\Delta(a) \cap \lk_\Delta(y)$ (which are walls in $\Delta$) increases the number of on-wall vertices with coefficient 0 in a wall relation in $\Delta$ using $a$ and $y$ as off-wall vertices by 1. \\

			\item The objective is to show that there is a hyperplane bounding a half-space from Lemma \ref{contractpointspace2} where both of the half-spaces bounded by this hyperplane are intersected there. A wall $\tau \in \lk_\Delta(m) \cap \lk_\Delta(n)$ with $m \in \{ a, b \}$ and $n \in \{ y, z \}$ is \emph{not} adjacent to $\{ y, z \} \setminus n$. This is implied by $\Delta$ being a simplicial pseudomanifold since each codimension 1 face is contained in exactly 2 facets (see Definition \ref{pseudomfldwallcross}). The flagness of $\Delta$ (implied by local convexity by Proposition 5.3 on p. 279 -- 280 of \cite{LR}) then forces the existence of ``bad'' vertices of $\tau$ that are not attached to $\{ y, z \} \setminus n$. Then, the vertex set of $\lk_\Delta(a) \cup \lk_\Delta(b)$ intersects both those of $\lk_\Delta(y) \setminus \lk_\Delta(z)$ and $\lk_\Delta(z) \setminus \lk_\Delta(y)$. This would mean that there are at least 2 distinct walls of $\Delta$ that are in $(\lk_\Delta(a) \cup \lk_\Delta(b)) \cap (\lk_\Delta(y) \cup \lk_\Delta(z))$. Then, Proposition \ref{convwallspan} implies that there is a pair of such walls of $\Delta$ sharing all but one vertex where one wall is opposite to $y$ and the other is opposite to $z$. After contraction of the edge $\{ y, z \}$, we end up with a pair of facets of $\widetilde{\Delta}$ sharing a wall that contain the contraction point $u$ of $\{ y, z \}$ as an on-wall vertex. The off-wall vertices of this wall crossing lie in $(\lk_\Delta(a) \cup \lk_\Delta(b)) \cap (\lk_\Delta(y) \cap \lk_\Delta(z))$. This implies that both half-spaces bounded by the hyperplane spanned by this wall are intersected in Lemma \ref{contractpointspace2} and intersection with the line segment $\overline{ab}$ has the properties claimed in the statement. \\

		\end{enumerate}
	\end{proof}

	Another instance where linear interpolations have this wide range of behaviors comes from certain external edge subdivisions. In fact, the same reasoning as Example \ref{contrcrosspolytop} (Lemma \ref{contractpointspace1}) describes shrinking contraction point spaces induced by certain external edge subdivisions although edge subdivisions seem to increase availability of convex wall crossings (Proposition \ref{edgesubdivwallconv}). We note that the reasoning for such cases extend beyond contraction point spaces realized by linear interpolations. As a whole, we have a mixture of expected and unexpected behavior induced by changes of the contraction point space of fixed edge $e = \{ a, b \} \in \Delta$ induced by external edge subdivisions. This depends on initial adjacency properties of the external edge of $\Delta$ contracted. \\

	\color{black}
	
	\begin{thm} \textbf{(External edge subdivisions vs. space of contraction points) \\} \label{exsubdivcon}
		Let $\Delta$ be a locally convex $(d - 1)$-dimensional simplicial pseudomanifold. In edge subdivisions/contractions, we will assume that vertices of $\Delta$ will keep the same $k^d$-coordinate representation unless they are one of the two vertices merged in an edge contraction. \\

		The (stellar) subdivision of an edge $\{ y, z \} \in \Delta$ can only affect the space of contraction points $w$ with respect to $e = \{ a, b \}$ if $y$ and $z$ are either in $\St_\Delta(a) \cup \St_\Delta(b)$ or opposite to $a$ or $b$ in a wall crossing of $\Delta$ using $a$ or $b$ as the initial off-wall vertex. We will assume that $|\St_\Delta(a)| \cup |\St_\Delta(b)|$ is convex so that the contraction point space of $e$ in $\Delta$ is not automatically empty by Lemma \ref{contractpointspace1}. Among such edges $\{ y, z \} \in \Delta$, we split into the following cases: \\
		\begin{enumerate}
			\item Suppose the subdivided edge $\{ y, z \} \in \Delta$ does \emph{not} lie in $|\St_\Delta(a)| \cup |\St_\Delta(b)|$. \\
			\begin{enumerate}
				\item If $y, z \notin \lk_\Delta(a) \cup \lk_\Delta(b)$, there is no change in the space of contraction points $w$ with respect to the edge $e = \{ a, b \}$. \\

				\item Assume the subdivided edge is of the form $\{ q, z \}$ for some $q \in \lk_\Delta(a) \cup \lk_\Delta(b)$. Then, the space of contraction points $w$ with respect to the edge $e = \{ a, b \}$ contains that of $\Delta$. \\
			\end{enumerate}
			
			\item Suppose the subdivided edge $\{ y, z \} \in \Delta$ lies inside $|\St_\Delta(a)| \cup |\St_\Delta(b)|$.
			\begin{enumerate}
				\item If $y, z \notin \{ a, b \}$, then the space of contraction points with respect to $w$ is contained in the intersection of the one for $\Delta$ and the intersection of spans of all walls containing the subdividing point (say $m$) and a vertex $a'$ or $b'$ opposite from $a$ or $b$. This comes from the wall containing the second off-wall vertex in a wall crossing containing $m$ in the wall that replaces $a$ with $a'$ or $b$ with $b'$. Thus, the linear contraction point space is the empty set, a single point, or all of $\overline{ab}$. \\
				
				\item Assume that the subdivided edge is of the form $\{ a, z \}$ or $\{ b, z \}$. If $z \in \lk_\Delta(a) \cap \lk_\Delta(b)$, the space of contraction points $w$ with respect to the edge $e = \{ a, b \}$ preserving convexity of wall crossings is empty. In general, the space of (linear) contraction points preserving convexity of wall crossings is empty, a single point, or the entire line segment $\overline{ab}$. \\

			\end{enumerate}
		\end{enumerate}

	\end{thm}
	
	\begin{proof}
		In each part, we will split into cases according to the type of external edge considered. The relevant properties of the edges will be whether the edge lies in $|\St_\Delta(a)| \cup |\St_\Delta(b)|$ and whether it uses one of $a$ or $b$ as its vertices. For inputs into Theorem \ref{contractpointspaceexpress}, we will keep track of changes in the half-spaces intersected in Lemma \ref{contractpointspace1} and Lemma \ref{contractpointspace2}. \\

		The initial assumptions on the subdivided edge $\{ y, z \} \in \Delta$ before splitting into cases filters out the edges that whose subdivision affects $\St_\Delta(a) \cup \St_\Delta(b)$ or edges of $\St_\Delta(q)$ attached to a vertex opposite to $a$ or $b$ in a wall crossing using $a$ or $b$ as an initial off-wall vertex for some $q \in \lk_\Delta(a) \cup \lk_\Delta(b)$. This is based on the parameters in the intersections from Theorem \ref{contractpointspaceexpress}.   \\
		\begin{enumerate}
		
			\item Suppose that the subdivided edge does \emph{not} lie in $|\St_\Delta(a)| \cup |\St_\Delta(b)|$. In order for the edge subdivision to possibly affect the contraction space for $e = \{ a, b \}$, the subdivided edge must lie in $\St_\Delta(q)$ for some $q \in \lk_\Delta(a) \cup \lk_\Delta(b)$. We note that $q \in \lk_\Delta(a) \setminus \lk_\Delta(b)$ or $q \in \lk_\Delta(b) \setminus \lk_\Delta(a)$ for vertices used for the half-space parameters in Lemma \ref{contractpointspace2} since facets of $\Delta$ containing both $a$ and $b$ lie in the \emph{interior} of $|\St_\Delta(a)| \cup |\St_\Delta(b)|$. We will further split into cases according to whether one of the vertices of the subdivided edge $\{ y, z \}$ is equal to a vertex $q \in \lk_\Delta(a) \cup \lk_\Delta(b)$ or not. In the cases below, the support of the union of the stars over $a$ and $b$ does \emph{not} change after subdividing $\{ y, z \}$. If this space is \emph{not} convex, the contraction point space of $\{ a, b \}$ is still empty even after subdividing $\{ y, z \}$. Thus, we will assume that $|\St_\Delta(a)| \cup |\St_\Delta(b)|$ in the cases below.  \\
			
			\begin{enumerate}
				\item If neither of the vertices of the subdivided edge $\{ y, z \} \in \Delta$ lie in $\lk_\Delta(a) \cup \lk_\Delta(b)$, then a facet $F \in \Delta$ with a vertex opposite to $a$ or $b$ cannot have $F \supset \{ y, z \}$. This is because a wall crossing replaces only one of the vertices and all vertices except the one opposite to $a$ or $b$ still lie in $\lk_\Delta(a) \cup \lk_\Delta(b)$. However, we assumed that neither $y$ nor $z$ lie in $\lk_\Delta(a) \cup \lk_\Delta(b)$. This means that none of the facets with vertices opposite to $a$ or $b$ are affected by the subdivision of the edge $\{ y, z \} \in \Delta$. The union $|\St_\Delta(a)| \cup |\St_\Delta(b)|$ also stays the same. Then Lemma \ref{contractpointspace1} and Lemma \ref{contractpointspace2} imply that there is no change in the space of contraction points $w$ with respect to the edge $e = \{ a, b \}$. \\  
				
				\item Now suppose that the subdivided edge is of the form $\{ q, z \}$ for some $q \in \lk_\Delta(a) \cup \lk_\Delta(b)$. Again, let $m$ be the subdividing vertex. As mentioned previously, the facets to consider are those coming from facets of $\Delta$ containing both $q$ and $z$. Since $\{ q, z \} \notin |\St_\Delta(a)| \cup |\St_\Delta(b)|$, there is no effect on $|\St_\Delta(a)| \cup |\St_\Delta(b)|$. If $z$ is \emph{not} opposite to $a$ or $b$, then Theorem \ref{contractpointspaceexpress} implies that there is no effect on the contraction point space. Thus, we will focus on the case where $z$ is opposite to $a$ or $b$. \\
				
				As mentioned above, the possible facets yielding altered half-spaces (see Lemma \ref{contractpointspace2}) are those involving $F \in \Delta$ such that $F \supset \{ q, z \}$. Since $z$ is opposite to $a$ or $b$ and $q \in \lk_\Delta(a) \cup \lk_\Delta(b)$, the new facet with a vertex opposite to $a$ is the one where $z$ is replaced by the subdividing vertex $m$. If we had replaced $q$ instead of $z$, the resulting facet would no longer share a wall with $\lk_\Delta(a) \cup \lk_\Delta(b)$. For walls inside $F$ that contain both $q$ and $z$, their counterparts in the edge subdivision with $m$ replacing $z$ do \emph{not} change the hyperplanes spanned by them since we assumed that $x_m = (1 - \alpha) x_q + \alpha x_z$ for some $\alpha \in (0, 1)$. This means that the sources of possible changes with respect to half-spaces to consider are from walls of the form $F \setminus q$ from facets $F \supset \{ q, z \}$. In other words, the possible changes come from convexity with respect to $q$. \\
				
				To study possible changes in convexity with respect to $q$, we will directly use properties of coefficients of ``wall relations'' where the second off-wall vertex is given by $a$ or $b$ in the wall crossing using $F$ as the initial facet and $z$ as the initial off-wall vertex. If $z$ is opposite to $a$ (respectively $b$), then the vertex $a$ (respectively $b$) automatically lies in the half-space associated to a vertex $r \in \lk_\Delta(a) \setminus \lk_\Delta(b)$ from Lemma \ref{contractpointspace2} by Proposition \ref{convsep}. Without loss of generality, suppose that $z$ is opposite to $a$. Our initial convexity assumptions on $\Delta$ imply that $|\St_\Delta(z)|$ is convex and $z \notin \Span(F \setminus z)$ since the realization of vertices of a face of a simplicial complex are linearly independent. Since the vertices of the realization of $F$ in $k^d$ are linearly independent and $F \supset \{ q, z \}$, we have a ``wall relation'' of the form \[ c_z x_z + c_b x_b + \sum_{p \in F \setminus z} c_p x_p = 0 \] with $c_z \ge 0$ and $c_b > 0$. This is because $a$ and $b$ are on the same side of the hyperplane $\Span(F \setminus z)$ by our assumption that $|\St_\Delta(a)| \cup |\St_\Delta(b)|$ is convex. We can rewrite this as \[ c_z x_z + c_b x_b + c_q x_q + \sum_{p \in F \setminus \{ z, q \} } c_p x_p = 0. \] 
				
				If $c_q = 0$, then replacing $a$ by $b$ yields a ``flat wall crossing'' with respect to $q$ and \emph{any} point on the line segment $\overline{ab}$ would yield a contraction point preserving convexity with respect to $q$. Suppose that $c_q \ne 0$. If we ``solve'' for $x_m$ via the substitution $x_z = \frac{1}{\alpha} x_m - \frac{1 - \alpha}{\alpha} x_q$ from $x_m = (1 - \alpha) x_q \alpha x_z$, we have

				\begin{align*}
					c_z \left( \frac{1}{\alpha} x_m - \frac{1 - \alpha}{\alpha} x_q \right) + c_b x_b + c_q x_q + \sum_{p \in F \setminus \{ z, q \} } c_p x_p &= 0 \\ 
					\Longrightarrow  \frac{c_z}{\alpha} x_m + c_b x_b + \left( c_q - \frac{1 - \alpha}{\alpha} c_z \right) x_q + \sum_{p \in F \setminus \{ z, q \} } c_p x_p &= 0.
				\end{align*}
				
				\color{black}
				The coefficient of $x_q$ becomes ``more negative'' since $c_z \ge 0$. Since the coefficient of $x_b$ is still $c_b$ (which is strictly positive), the coefficient of $x_q$ in the expansion of $x_b$ is more negative with respect to the $((F \setminus z) \cup m)$-basis expansion of $k^d$ than the $F$-basis expansion. Note that $c_q \le 0$ when we use $a$ in place of $b$ since $a$ respects convexity with respect to $q$. The decrease in the former coefficient implies that a larger portion of the line segment $\overline{ab}$ (i.e. more values $\eta \in (0, 1)$) yielding suitable $(1 - \eta) x_a + \eta x_b$) compatible with convexity with respect to $q$. When $\eta \in (0, 1)$, using $x_w = (1 - \eta) x_a + \eta x_b$ in place of $x_z$ above gives a wall relation with strictly positive coefficients of $x_w$ and $x_b$. This can be expressed more concretely in terms of wall relations. Let \[ \alpha_m x_m + \alpha_a x_a + \alpha_q x_q + \sum_{p \in F \setminus \{ z, q \}} \alpha_p x_p = 0 \] be the wall relation in $\Delta'$ with $\alpha_m > 0$, $\alpha_a > 0$, $\alpha_q \le 0$, and $\alpha_p \le 0$ for $p \in F \setminus \{ z, q \}$ (see Proposition \ref{edgesubdivwallconv}). Since \[ \frac{c_z}{\alpha} x_m + c_b x_b + \left( c_q - \frac{1 - \alpha}{\alpha} c_z \right) x_q + \sum_{p \in F \setminus \{ z, q \} } c_p x_p = 0, \] multiplying the first relation by $\frac{1 - \eta}{\alpha_a}$ and the second one by $\frac{\eta}{c_b}$ before adding them together implies that \[ \left( \frac{(1 - \eta) \alpha_m}{\alpha_a} + \frac{\eta c_z}{\alpha c_b} \right) x_m + x_w + \left( \frac{(1 - \eta) \alpha_q}{\alpha_a} + \frac{\eta \left( c_q - \frac{1 - \alpha}{\alpha} c_z \right) }{c_b} \right) x_q + \sum_{p \in F \setminus \{ z, q \}} \left( \frac{(1 - \eta) \alpha_p}{\alpha_a} + \frac{\eta c_p}{c_b} \right) x_p = 0 \] in the edge contraction $\widetilde{\Delta}$ since $x_w = (1 - \eta) x_a + \eta x_b$.   \\

			Finally, we note that the subdividing point $m$ of $\{ y, z \}$ does \emph{not} occur as an on-wall vertex of a wall crossing involving the contraction point $w$ with respect to $e = \{ a, b \}$. This is because $\{y, z \} \notin \St_\Delta(a) \cup \St_\Delta(b)$. Then, the convexity of wall crossings with respect to $m$ is unaffected by the choice of contraction point $w$ of $e = \{ a, b \}$. The computations in Proposition \ref{edgesubdivwallconv} imply that they are always convex when we take $x_m = (1 - \alpha) x_q + \alpha x_z$ for some $\alpha \in (0, 1)$. \\

		\end{enumerate}

		\item Now suppose that the subdivided edge $\{ y, z \}$ \emph{does} lie in $|\St_\Delta(a)| \cup |\St_\Delta(b)|$. We will split into cases according to whether one of the vertices is in $\{ a, b \}$ or not. Assume that $|\St_\Delta(a)| \cup |\St_\Delta(b)|$ is convex. \\
		
		\begin{enumerate}
			\item Suppose that the subdivided edge does \emph{not} involve $a$ or $b$. Since $x_m = (1 - \alpha) x_y + \alpha x_z$ for some $\alpha \in (0, 1)$, we have that $|\St_{\Delta'}(a)| \cup |\St_{\Delta'}(b)| = |\St_\Delta(a)| \cup |\St_\Delta(b)|$ is a convex set. By Proposition \ref{starunionadj}, $a$ or $b$ must be a common vertex that $y$ and $z$ are both adjacent to in $\Delta$. Without loss of generality, suppose that this vertex is $b$. The half-space changes come from walls of facets containing $\{ y, z \}$ that contain an off-wall vertex $a'$ or $b'$ replacing $a$ or $b$ in a wall crossing. Note that the edge $\{ y, z \} \in \Delta$ must lie on the wall in such a wall crossing. Since $y, z \in \lk_\Delta(b)$, we need to intersect half-spaces pointing in both directions of each wall containing the subdividing vertex $m$ and $b'$. This means that both $y$ and $z$ play the role of $r$ in Lemma \ref{contractpointspace2} and we need to consider the intersection of their spans. Since the sets intersected include hyperplanes, the intersection with the line segment $\overline{ab}$ is the empty set, a single point, or all of $\overline{ab}$. \\

			\item Now assume that the subdivided edge in $|\St_\Delta(a)| \cup |\St_\Delta(b)|$ is of the form $\{ a, z \}$ or $\{ b, z \}$. Without loss of generality, suppose that it is $\{ a, z \}$ for some $z \in \lk_\Delta(a)$. \\
			
			In certain cases, it may be possible to have failure of convexity of $|\St_{\Delta'}(a)| \cup |\St_{\Delta'}(b)|$. Suppose that $z \in \lk_\Delta(a) \cap \lk_\Delta(b)$. In the edge subdivision $\Delta'$ with respect to $\{ a, z \}$, we have that $z \in \lk_{\Delta'}(b) \setminus \lk_{\Delta'}(a)$. Let $F \in \Delta$ be a facet of $\Delta$ containing $a$, $b$, and $z$. Since $\Delta$ is a simplicial pseudomanifold, there is exactly one vertex $s \in V(\Delta)$ that such that $s \ne b$ and $(F \setminus b) \cup s \in \Delta$. Consider the line segment $\overline{zs}$. While $z \in \lk_{\Delta'}(b)$ and $s \in \lk_{\Delta'}(a)$, we have that $\{ z, s \} \notin \St_{\Delta'}(a) \cup \St_{\Delta'}(b)$. If $|\St_{\Delta'}(a)| \cup |\St_{\Delta'}(b)|$ were convex, this would imply that the line segment $\overline{zs} \subset |\St_{\Delta'}(a)| \cup |\St_{\Delta'}(b)|$. However, points on the interior of the line segment $\overline{zs}$ do \emph{not} lie in $|\St_{\Delta'}(a)| \cup |\St_{\Delta'}(b)|$ since the coefficient of $a$ must be at least as large as that of $z$ in the $F$-basis expansion of $k^d$-points in the realization of $(F \setminus z) \cup m \in \Delta'$ (a convex linear combination of its vertices). This implies that $|\St_{\Delta'}(a)| \cup |\St_{\Delta'}(b)|$ is \emph{not} convex since $z \in \lk_{\Delta'}(b) \setminus \lk_{\Delta'}(a)$ and $s \in \lk_{\Delta'}(a) \setminus \lk_{\Delta'}(b)$. Thus, it is not possible for contraction of the edge $\{ a, z \} \in \Delta$ to preserve convexity of wall crossings if $z \in \lk_\Delta(a) \cap \lk_\Delta(b)$. \\ 
			
			From now on, we will assume that $z \in \lk_\Delta(a) \setminus \lk_\Delta(b)$. This means that $z \notin |\St_{\Delta'}(a)| \cup |\St_{\Delta'}(b)|$. Since $z \notin \lk_\Delta(b)$, faces that are adjacent to both $a$ and $z$ in $\Delta$ are \emph{not} adjacent to $b$. Consider wall crossings in $\Delta'$ using $z$ and $a$ as off-wall vertices. Given a wall $\tau \in \Delta$ such that $\tau \supset \{ a, z \}$, the wall $(\tau \setminus a) \cup m \in \Delta'$ is attached to exactly two vertices of $\Delta'$ (which were vertices of $\Delta$). Both facets of $\Delta'$ containing $(\tau \setminus a) \cup m$ are in $\St_\Delta(a) \setminus \St_\Delta(b)$. This is because faces that are adjacent to both $a$ and $z$ in $\Delta$ are \emph{not} adjacent to $b$ and $|\St_\Delta(a)| \cup |\St_\Delta(b)|$ is taken to be convex (see Proposition \ref{starunionadj}). Both half-spaces bounded by the hyperplane $\Span((\tau \setminus a) \cup m)$ are intersected since $\alpha, \beta \in \lk_\Delta(a) \setminus \lk_\Delta(b)$. In other words, $\alpha$ and $\beta$ can both take the role of $r \in \lk_\Delta(a) \cup \lk_\Delta(b)$ from Lemma \ref{contractpointspace2}. The intersection of the hyperplane $\Span((\tau \setminus a) \cup m)$ with the line segment $\overline{ab}$ is either the entire line segment $\overline{ab}$, a single point, or the empty set.

		\end{enumerate}

	\end{enumerate}

	\end{proof}

	The other side of the snapshots of changes of the contraction point space of $e = \{ a, b \} \in \Delta$ from PL homeomorphisms induced by external edges comes from edge contractions. Here, there can be a wide range of outcomes even within the same adjacency conditions. This may be a reflection of PL homeomorphisms being a topological equivalence although we do have compatibility with the simplicial complex structure. The first parameter to consider is whether hyperplanes bounding half-spaces from Theorem \ref{contractpointspaceexpress} separate the vertices $a$ and $b$. If they do, the specific choice of contraction point space is important. It involves comparing sizes of projections to these hyperplanes. Normalizing for sizes of $k^d$-vectors realizing the vertices would be a comparison of the angles made with the hyperplanes.  \\
	
	Before giving a specific example illustrating the considerations discussed above (Corollary \ref{abedgeconlin}), we will first organize them into a general framework for analyzing such cases.  \\

	\begin{prop} \textbf{(Framework for external edge contractions) \\}
		\label{genextcon}
		
		Let $\Delta$ be a locally convex $(d - 1)$-dimensional simplicial pseudomanifold. As in Proposition \ref{exsubdivcon}, assume that vertices of $\Delta$ will keep the same $k^d$-coordinate representation unless they are one of the two vertices merged in an edge contraction. \\

		Consider contractions of edges $\{ y, z \} \in \Delta$ to points $u$ that preserve the convexity of supports of stars over each vertex. Assume that the external contractions involved are linear interpolations with $x_u = (1 - \eta) x_y + \eta x_z$ in $k^d$ for some $\eta \in (0, 1)$. \\ 
		
		\begin{enumerate}
			\item \textbf{(Nontrivial external edges) \\}
			The contraction $\widetilde{\Delta}$ of the edge $\{ y, z \} \in \Delta$ can only affect the space of contraction points $w$ with respect to $e = \{ a, b \}$ only if at least one of $y$ and $z$ are in $\St_\Delta(a) \cup \St_\Delta(b)$ or opposite to $a$ or $b$ in a wall crossing of $\Delta$ using $a$ or $b$ as the initial off-wall vertex. \\
			
			\item Suppose that the conditions of Part 1 are satisfied. We split the effect on the contraction point space with respect to $e = \{ a, b \}$ in $\widetilde{\Delta}$ into following adjacency properties. Note that there may be a variance in behavior within the same adjacency case depending on separation and angle/projection properties. \\
			\begin{enumerate}
				
				\item \textbf{(Initial ambient convexity conditions) \\}
				Nonemptiness requires convexity of $|\St_{\widetilde{\Delta}}(a)| \cup |\St_{\widetilde{\Delta}}(b)|$. Note that this is equal to $|\St_\Delta(a)| \cup |\St_\Delta(b)|$ if $y, z \notin \St_\Delta(a) \cup \St_\Delta(b)$. If $y \in \St_\Delta(a) \cup \St_\Delta(b)$ or $z \in \St_\Delta(a) \cup \St_\Delta(b)$, then contraction of $\{ y, z \}$ to a point in the interior of the line segment $\overline{yz}$ implies that $|\St_{\widetilde{\Delta}}(a)| \cup |\St_{\widetilde{\Delta}}(b)| \ne |\St_\Delta(a)| \cup |\St_\Delta(b)|$. For example, instances where $y = a$ or $y = b$ yield initial convexity constraints since we need $|\St_\Delta(a)| \cup |\St_\Delta(b)| \cup |\St_\Delta(z)|$ to be convex. The space of suitable $k^d$-coordinates for $u$ yielding ``convex completions'' in such cases is defined by inequalities analogous to those in Lemma \ref{contractpointspace2}.  \\

				\item \textbf{(Contraction points as off-wall vertices opposite to $a$ or $b$) \\}
				Suppose $y$ (respectively $z$) is the off-wall vertex opposite to $a$ (respectively $b$) via a single wall crossing from a facet $F \in \St_\Delta(a)$ (respectively $F \in \St_\Delta(b)$). In $\widetilde{\Delta}$, walls of $(F \setminus a) \cup y$ or $(F \setminus a) \cup z$ (respectively $(F \setminus b) \cup y$ or $(F \setminus b) \cup z$ ) that contain $y$ or $z$ are replaced by $(F \setminus a) \cup u$ (respectively $(F \setminus b) \cup u$). The half-spaces bounded by $\Span ((F \setminus a) \cup u)$ (respectively $\Span((F \setminus b) \cup u)$) in Lemma \ref{contractpointspace2} contain $a$ (respectively $b$). \\
				
				This follows from the compatibility of contraction $\{ y, z \} \in \Delta$ with convexity of wall crossings. For half-spaces bounded by specific walls, the convexity is considered with respect to the on-wall vertex in $\lk_\Delta(a)$ (respectively $\lk_\Delta(b)$) that is removed from the facets containing $y$ or $z$ above to form the wall. \\

				\item \textbf{(Contraction points as on-wall vertices in $\lk_\Delta(a)$ or $\lk_\Delta(b)$) \\}
				If $y \in \lk_\Delta(a) \cup \lk_\Delta(b)$ or $z \in \lk_\Delta(a) \cup \lk_\Delta(b)$, the change is in on-wall vertices of wall crossings involving facets of the form $G \in \St_\Delta(a)$ (respectively $G \in \St_\Delta(b)$) and $(G \setminus p) \cup x$ for some $x \ne y, z$ opposite to $a$ (respectively $b$), where $G \ni y$ (respectively $G \ni z$) and $p \in \lk_\Delta(a) \cup \lk_\Delta(b)$ is an on-wall vertex such that $p \ne y$ (respectively $p \ne z$). In $\widetilde{\Delta}$, $y$ and $z$ are replaced by the contraction point $u$ of $\{ y, z \}$ to obtain the facet $(G \setminus y) \cup u$ (respectively $(G \setminus z) \cup u$) replacing $G$ in $\widetilde{\Delta}$. \\ 
				
				The wall crossing considered here involves the facets $G$ and $(G \setminus x) \cup a$ (respectively $(G \setminus x) \cup b$). The change is in convexity with respect to vertices $r \ne u$ on the wall $G \setminus x$. Note that the changes in ``off-wall convexity properties'' (see Lemma \ref{contractpointspace2}) entirely result from changes in the $k^d$-coordinate from substituting the vertex $u$ in place of $y$ or $z$ instead of a change in the simplicial complex structure of the relevant facets. \\
				
				Fewer half-spaces from Lemma \ref{contractpointspace2} are intersected compared to $\Delta$ since those involving facets of $\Delta$ containing $\{ y, z \}$  are no longer intersected since $\widetilde{\Delta}$ does \emph{not} contain any facets corresponding to those of $\Delta$ containing the edge $\{ y, z \} \in \Delta$. \\
				
				\item \textbf{(Contracting edges containing $a$ or $b$) \\}
				If $y \in \{ a, b \}$ or $z \in \{ a, b \}$, the situation is similar to Part c) except that the half-spaces intersected to form the contraction point space of $e = \{ a, b \}$ in $\Delta$ all have counterparts in $\widetilde{\Delta}$. The difference is that there is a smaller change in $|\St_{\widetilde{\Delta}}(a)| \cup |\St_{\widetilde{\Delta}}(b)|$ compared to $|\St_\Delta(a)| \cup |\St_\Delta(b)|$ and that there are fewer half-spaces for $\widetilde{\Delta}$ from Lemma \ref{contractpointspace2} where the substitution of the contraction point $u$ of $\{ y, z \}$ in place of $y$ or $z$ induces a change in the half-space. \\

				\item \textbf{(Linear interpolation and (non)separation of vertices by convexity conditions) \\}
				Consider the hyperplanes spanned by modified walls in $\widetilde{\Delta}$ bounding the half-spaces defining convexity with respect to a fixed on-wall vertex (Proposition \ref{convsep}). Hyperplane (non)separation properties and projection/angle comparisons yield variance in behavior within the same adjacency properties from the abstract simplicial complex structure of $\Delta$. If they do \emph{not} separate $a$ and $b$, then the entire line segment $\overline{ab}$ lies in the contraction point space of $e = \{ a, b \}$ in $\widetilde{\Delta}$. Otherwise, the situation depends on the effect of replacing $a$ (respectively $b$) with $b$ (respectively $a$) and whether hyperplanes bounding half-spaces defining convexity conditions from Lemma \ref{contractpointspace2} separate $y$ and $z$ or not. The latter follows from the contraction point $u$ of $\{ y, z \}$ being a linear interpolation between $y$ and $z$. \\
			\end{enumerate}
		\end{enumerate}
	\end{prop}

	We now give an example of the variance in behavior induced by external edge contractions within the same adjacency conditions. \\
	
	\color{black}
	
	\begin{cor} \textbf{(Linear interpolation and contracting edges containing $a$ or $b$) \\} \label{abedgeconlin}
		Suppose that the contracted edge is $\{ b, z \} \in \Delta$. Assume that contraction of this edge preserves convexity of wall crossings. Let $u$ be the contraction point of $\{ b, z \}$ and $\widetilde{\Delta}$ be the simplicial complex obtained after contraction of $\{ b, z \} \in \Delta$. Note that the same statements hold with $a$ and $b$ switched if $\{ a, z \}$ is contracted instead. \\
		
		\begin{enumerate}
			\item The convexity of wall crossings in $\widetilde{\Delta}$ with respect to $u$ is equivalent to convexity of $|\St_{\widetilde{\Delta}}(a)| \cup |\St_\Delta(b)| \cup |\St_\Delta(z)|$. If $z \in \lk_\Delta(a) \cap \lk_\Delta(b)$, then the contraction point $u$ needs to satisfy special alignment properties. The space of suitable points is parametrized by intersections of half-spaces analogous to those in Lemma \ref{contractpointspace2}. They describe ``completions'' to convex sets after removing a single vertex of the simplicial complex and replacing it with another point in the realization space. \\
			
			\item Consider facets $G \in \Delta$ such that $G \supset \{ b, z \}$ and wall crossings where $b$ is replaced by a vertex $b' \notin \St_\Delta(a) \cup \St_\Delta(b)$ to obtain a facet $(G \setminus b) \cup b'$. Suppose that a hyperplane of the form $\Span(G \setminus \{ b, r \}, b')$ with $r \in G \setminus b$ and $r \ne z$ separates vertices $a$ and $b$ for some on-wall vertex $r \in G \setminus b$. Otherwise, any point on the line segment $\overline{ab}$ would preserve convexity of the induced wall crossing with respect to $r$ in $\widetilde{\Delta}$. In other words, we are assuming that switching $b$ with $a$ does \emph{not} preserve the convexity of this wall crossing with respect to $r$ in $\Delta$. Assume that the same is true after $b$ is replaced by the contraction point $u$ of $\{ b, z \}$. \\
			
			Whether contraction of $\{ b, z \}$ makes $a$ and $b$ pushes this ``convexity with respect to $r$ hyperplane'' closer to $a$ or $b$ is determined by the sign of the $x_z$-coefficient of the intersection point with $\overline{ab}$ as an element of $\Span(G \setminus \{ b, r \}, b')$. Nonnegativity means it would be pushed closer to $b$ and negativity means it is pushed closer to $a$. The latter case must satisfy a ``lower bound of ratios'' in order to be compatible with separation of $a$ and $b$ by this hyperplane. The parameters involved have to do with projection/angle data with respect to hyperplanes that bound half-spaces defining convexity properties. \\
		\end{enumerate}
	\end{cor}
	
	\begin{proof}

		\begin{enumerate}
			\item As a start, the nonemptiness of the contraction point space with respect to $\{ a, u \} \in \widetilde{\Delta}$ ($u$ replacing $b$) in the contraction $\widetilde{\Delta}$ of the edge $\{ b, z \}$ requires convexity of $|\St_{\widetilde{\Delta}}(a)| \cup |\St_{\widetilde{\Delta}}(u)| = |\St_{\widetilde{\Delta}}(a)| \cup |\St_\Delta(b)| \cup |\St_\Delta(z)|$. If $z \in \lk_\Delta(a) \cap \lk_\Delta(b)$, the contraction point $u$ of $\{ b, z \} \in \Delta$ may need to satisfy special conditions since lines between a vertex of $\lk_\Delta(a) \setminus \lk_\Delta(b)$ and a vertex of $\lk_\Delta(b) \setminus \lk_\Delta(a)$ must still stay inside $|\St_{\widetilde{\Delta}}(a)| \cup |\St_{\widetilde{\Delta}}(b)|$. The connection to Lemma \ref{contractpointspace2} comes from studying when replacement of the single vertex $z$ by $u$ in $\St_\Delta(a) \cup \St_\Delta(b)$ means $|\St_{\widetilde{\Delta}}(a)| \cup |\St_{\widetilde{\Delta}}(u)|$ is convex. These cases indicate whether the ``wall crossing'' involving $(G \setminus b) \cup a$ and $(G \setminus b) \cup b'$ instead of $G$ and $(G \setminus b) \cup b'$ is closer or further away from being convex with respect to $r$. \\ 
			
			\item Our initial assumption that contraction of $\{ b, z \}$ preserves convexity of wall crossings implies that the half-space associated to this wall from Lemma \ref{contractpointspace2} contains $b$. Note this defines convexity of the wall crossing with respect to $r$. Such a separation means that there is a parameter $\eta \in (0, 1)$ such that there is a linear relation of the form \[ \sum_{m \in (G \setminus \{ b, r \}) \cup b'} c_m x_m = \sum_{m \in (G \setminus \{ b, r, z \}) \cup b'} c_m x_m + c_z x_z  =  (1 - \eta) x_a + \eta x_b. \] Note that $z \in G \setminus \{ b, r \}$ since $G \supset \{ b, z \}$ and $r \ne z$. \\
			
			After contracting the edge $\{ b, z \} \in \Delta$ to a point $u$ that is a linear interpolation with \[ x_u = (1 - \alpha) x_b + \alpha x_z \] for some $\alpha \in (0, 1)$, substituting \[ x_z = \frac{1}{\alpha} x_u - \frac{1 - \alpha}{\alpha} x_b \] into the identity above implies that

			\begin{align*}
				\sum_{m \in (G \setminus \{ b, r, z \}) \cup b'} c_m x_m + c_z x_z  &=  (1 - \eta) x_a + \eta x_b \\
				\Longrightarrow \sum_{m \in (G \setminus \{ b, r, z \}) \cup b'} c_m x_m + c_z \left( \frac{1}{\alpha} x_u - \frac{1 - \alpha}{\alpha} x_b \right)  &=  (1 - \eta) x_a + \eta x_b \\
				\Longrightarrow \sum_{m \in (G \setminus \{ b, r, z \}) \cup b'} c_m x_m + \frac{c_z}{\alpha} x_u - \frac{c_z(1 - \alpha)}{\alpha} x_b &= (1 - \eta) x_a + \eta x_b \\
				\Longrightarrow \sum_{m \in (G \setminus \{ b, r, z \}) \cup b'} c_m x_m + \frac{c_z}{\alpha} x_u &= (1 - \eta) x_a + \left( \eta + \frac{c_z(1 - \alpha)}{\alpha} \right) x_b.
			\end{align*}

			Note that the linear relation above among the vertices in $(G \setminus \{ b, z \}) \cup \{ u, a \}$ is unique up to scaling since the $k^d$-coordinates vertices of the realization of a facet of a $(d - 1)$-dimensional simplicial complex are linearly independent. \\
			
			Since $1 - \eta > 0$, separation of the vertices $a$ and $b$ (i.e. $a$ failing to induce convexity with respect to $r$) by the span of the modified wall after contraction of $\{ b, z \}$ means that we need \[ \eta + \frac{c_z(1 - \alpha)}{\alpha} > 0 \Longleftrightarrow c_z > -\frac{\eta \alpha}{1 - \alpha}. \] 
			
			If $c_z \ge 0$, then point of intersection with the interior of the line segment $\overline{ab}$ is closer to $b$ (i.e. $a$ and $b$ ``more separated''). Otherwise, $c_z < 0$ and the point of intersection with the interior of the line segment $\overline{ab}$ is closer to $a$ (i.e. $a$ and $b$ ``less separated''). \\
			
			Finally, the last statement is an application of Proposition \ref{ratwallconv} and Proposition \ref{convsep}.

		\end{enumerate}

	\end{proof}

	\subsection{Boundary behavior} \label{bdrybeh}
	
	In this subsection, we focus on boundary conditions for (local) convexity. In particular, we consider flat wall crossings (Part 3 of Proposition \ref{ratwallconv}) where the hyperplane in the convexity condition from Proposition \ref{convsep} contains both off-wall vertices (Proposition \ref{convsep}). These are wall crossings that ``just manage'' to be convex. The combinatorial objects we compare flat wall crossings with are induced 4-cycles in the 1-skeleton of the given simplicial complex. We discuss the combinatorial and geometric motivation leading to the 4-cycles below. \\
	
	Given close connections between convex wall crossings and combinatorial positivity properties coming from flag simplicial complexes (\cite{LR}, \cite{Psig}, \cite{Athgam}), it is natural to look at ``boundary cases'' of flag simplicial pseudomanifolds. Apart from this, we note that the condition $\lk_\Delta(a) \cap \lk_\Delta(b) = \lk_\Delta(e)$ for edges $e = \{ a, b \} \in \Delta$ (often used when $\Delta$ is flag) is equivalent to the contraction of $e$ preserving the PL homeomorphism type \cite{Nev}. Regarding changes under PL homeomorphisms, we start by noting that edge subdivisions preserve the flag property and increase the space of convex wall crossings while preserving convexity of the old ones (Proposition \ref{edgesubdivwallconv}). This suggests a comparison with simplicial pseudomanifolds minimal with respect to edge subdivisions (i.e. edge contractions failing to preserve flagness). These minimal cases are those where the 1-skeleton of the given simplicial complex is covered by induced 4-cycles (p. 79 of \cite{LN}). In other words, the other simplicial pseudomanifolds in the same PL homeomorphism class are edge subdivisions of these ``minimal building blocks''. Since edge subdivisions increase space of convex wall crossings and preserve old ones (Proposition \ref{edgesubdivwallconv}) and edge subdivisions preserve flagness, these are natural points of comparison. \\
	
	In this context, we compare effects of PL homeomorphisms on induced 4-cycles in the 1-skeleton and flat wall crossings. The former may be taken as a proxy for ``intrinsic'' structures from abstract simplicial complex structures where we find interesting parallels. The coordinate realizations play a role in degenerate linear conditions such as those on maintaining flatness of wall crossings after a coordinate change in a single vertex. \\

	\begin{prop} \textbf{(Changes in induced 4-cycles vs. flatness of wall crossings) \\} \label{bdry4cycleflat}
		
		Suppose that $\Delta$ is a simplicial sphere lying in $k^d$ that is locally convex with convex wall crossings (in the sense of p. 420 -- 421 of \cite{Mat}). We will take the usual realization of edge subdivisions/contractions as linear interpolations that keep all $k^d$-realizations of all unaltered vertices the same as those in the initial simplicial complex. While we focus on newly formed induced 4-cycles and flat wall crossings, we will make some comments on existing ones before the transformations are made. \\
		
		\begin{enumerate}
			\item \textbf{(Suspensions) \\}
			Suspensions form new 4-cycles exactly when they form new wall relations that are flat with respect to on-wall vertices. They preserve existing 4-cycles and flat wall crossings. \\
			
			\item \textbf{(Edge subdivisions)}
				\begin{enumerate}
					\item The new 4-cycles from edge subdivisions are from suspensions of non-adjacent vertices in $\lk_\Delta(e)$ by $a$ and $b$. Existing 4-cycles are preserved unless they use $e$ as an edge. 
					
					\item The new flat wall crossings are those that use the subdividing vertex $v$ as an on-wall vertex and are flat with respect to the remaining on-wall vertices (which are in $\lk_\Delta(e)$). Both settings have objects removed from $\{ a, b \}$ no longer being an edge in the edge subdivision and need the entire edge $e = \{ a, b \}$ involved in order to produce a change.
					
					\item For counterparts of wall crossings in $\Delta$, the only possible changes come from wall crossings in $\Delta'$ involving $v$. In such a wall crossing, flatness is only preserved with respect to the vertices excluding the vertex of $a$ or $b$ that is \emph{not} replaced by $v$. \\
				\end{enumerate}

			\item \textbf{(Edge contractions) \\}
				\begin{enumerate}
					\item The new 4-cycles from edge contractions are edge contractions of older 5-cycles (which themselves are only produced by edge subdivisions/contractions and not suspensions). Existing 4-cycles are preserved unless they use the edge $e$. \\
					
					\item Consider the new flat wall crossings produced from edge contractions. More specifically, we look at the flatness property that only depends on changes in the simplicial complex structure and convexity properties of the previous step. For such wall crossings, this is flatness with respect to the contraction point $w$ of $\{a, b\}$. \\
					
					Such new flat wall crossings are produced from pairs of consecutive wall crossings $F \cup G$ and $G \cup H$ in $\Delta$ where the first wall crossing is strictly convex with respect to the on-wall vertex $a$ (respectively $b$) and the second wall crossing is strictly convex with respect to the on-wall vertex $b$ (respectively $a$). The vertex $a$ (respectively $b$) is off-wall in the second wall crossing. \\
					
					\item The use of strictly convex wall crossings to produce a flat wall crossing is analogous to edge subdivisions of 4-cycles being needed to produce 4-cycles via an edge contraction. Note that edge subdivisions make coefficients of vertices in unscaled wall relations with the usual signs more negative (Proposition \ref{edgesubdivwallconv}). \\

					\item As for counterparts of existing flat wall crossings in $\Delta$, they stay the same if they do not involve $a$ or $b$. However, wall crossings involving exactly one of $a$ or $b$ producing or preserving flatness with respect to a vertex $p$ may require certain ``degenerate'' conditions. \\
					
					For counterparts of wall crossings in $\Delta$ involving $a$ or $b$, flatness with respect to a specific on-wall vertex $p$ means containment in a specific hyperplane (see Proposition \ref{convsep} and Proposition \ref{ratwallconv}). There is at most one such point on the line segment $\overline{ab}$ unless the line spanned by $a$ and $b$ is contained in this hyperplane. The condition is more restrictive when we consider flatness of multiple such wall crossings. Note that this does \emph{not} involve a change in the simplicial complex structure and only the $k^d$-coordinate choice for $a$ or $b$. \\
					
					\color{black}
					
				\end{enumerate}

		\end{enumerate}
	\end{prop}
	
	\begin{proof}
		We will compare the effect of suspensions and edge subdivisions/contractions on induced 4-cycles and flatness of wall crossings. Note that an induced 4-cycle can be interpreted as the suspension of a pair of non-adjacent vertices. The starting simplicial complex for each transformation will be denoted $\Delta$. \\
		
		\begin{enumerate}
			\item A suspension by a pair of vertices $d$ and $-d$ yields new induced 4-cycles coming from non-adjacent pairs of vertices in the previous step. Note that each suspension vertex is adjacent to every vertex of the simplicial complex we take the suspension over. \\
			
			From the perspective of wall relations, the relation $x_d + x_{-d} = 0$ ($x_p$ being the $k^d$-coordinate of a vertex $p$) is a new wall relation where the remaining vertices on the wall have coordinate 0. In this sense, a wall crossing in the suspension formed by replacing $d$ by $-d$ is flat with respect to \emph{any} on-wall vertex. From the non-separating hyperplane perspective, any hyperplane that contains $x_d$ also contains $x_{-d}$ since it is a vector space. This means that $x_d$ and $x_{-d}$ are \emph{not} separated by hyperplanes given by the span of $d$ and $\tau \setminus a$ for the vertex we measure convexity with respect to and we have a flat wall crossing. For existing wall crossings where $x_d$ or $x_{-d}$ is on-wall, the new induced wall crossings are flat with respect to the suspension vertex used ($d$ or $-d$) and the flatness/convexity properties of the remaining on-wall vertices are unaffected. \\

			\item Next, we consider the effect of (stellar) subdivision of an edge $e = \{ a, b \} \in \Delta$. The induced 4-cycles in $\Delta$ \emph{not} involving $e = \{ a, b \}$ are preserved. On the other hand, the induced $4$-cycles in $\Delta$ that \emph{do} use $e = \{ a, b \}$ are removed since $a$ and $b$ are no longer adjacent (form $5$-cycles instead). Since the vertices $a$ and $b$ are no longer adjacent in the edge subdivision $\Delta'$, we obtain new induced 4-cycles from the suspension of non-adjacent pairs of vertices in $\lk_\Delta(e)$ by the vertices $a$ and $b$. \\
			
			For flatness/convexity of wall crossings, we note that the wall relation $x_a + x_b - 2x_v = 0$ accounts for wall crossings in $\Delta'$ using the subdividing vertex $v$ as an on-wall vertex. This indicates that such wall crossings are flat with respect to the rest of the on-wall vertices. By Proposition \ref{edgesubdivwallconv}, wall crossings in $\Delta'$ using $v$ as an off-wall vertex maintain their convexity and do \emph{not} introduce any new flat wall crossings. Note that wall crossings in $\Delta$ where the union of the facets does \emph{not} contain $e = \{ a, b \}$ remain the same in $\Delta'$. The ones that do are removed and replaced with the wall crossings involving $v$ which we discussed earlier. Thus, both the new induced 4-cycles and the new flat wall crossings are induced by suspension over the vertices $a$ and $b$. \\

			As for counterparts of existing wall crossings of $\Delta$, the only possible changes come from wall crossings of $\Delta'$ involving $v$. To look for changes in flat wall crossings, we look at cases where $v$ is an on-wall or off-wall vertex of a wall crossing in the proof of Proposition \ref{edgesubdivwallconv}. Note that $v$ can only replace one of $a$ or $b$ in either facet of a given wall crossing of $\Delta$. This is something we can show directly when $v$ is an on-wall vertex. If $v$ is an off-wall vertex, the vertex replacing it is \emph{not} equal to $a$ or $b$ and we get a facet of $\Delta$ that does \emph{not} contain $e = \{ a, b \}$. While it is shown that convexity of wall crossings is preserved, we want to look for possible instances where there is a strict decrease in the coefficients involved that forces strict negativity of the wall crossing coefficients. In both the on-wall and off-wall cases of $v$, the only vertices where this can occur are the vertex in $e = \{ a, b \}$ that is \emph{not} replaced by $v$ in the given wall crossing. \\

			\item Finally, we study contractions $\widetilde{\Delta}$ of an edge $e = \{ a, b \} \in \Delta$. As in the edge subdivision case, the induced $4$-cycles \emph{not} using $e = \{ a, b \}$ are preserved and those that \emph{do} use $e = \{ a, b \}$ are removed (turned into 3-cycles in this case). Any new $4$-cycles formed in $\widetilde{\Delta}$ are from $5$-cycles of $\Delta$. One possible source of an induced $5$-cycle is from subdivision of an edge $\{ p, q \}$ of an induced $4$-cycle. This is because the link over the subdividing vertex $v$ is the suspension over $p$ and $q$ of the link over $\{ p, q \}$ and the other two vertices of an induced 4-cycle containing $\{ p, q \}$ do \emph{not} lie in the link over $\{ p, q \}$. Also, suspensions \emph{cannot} be a source of new induced 5-cycles since a suspension vertex is adjacent to every vertex except the other suspension point. However, any vertex of an induced 5-cycle is \emph{not} adjacent to 2 of the other other vertices in the induced 5-cycle. Thus, the only source of induced 5-cycles is from edge subdivisions of induced $4$-cycles or edge contractions of induced $6$-cycles. These cases both fall into edge subdivisions/contractions of edges in induced 4-cycles from an earlier stage. \\
			
			We will compare this to the effect on flatness of wall crossings. The focus will be on changes in simplicial complex structures since we are making comparisons with structural changes from induced 4-cycles in the 1-skeleton. As in the edge subdivision case, wall crossings that do \emph{not} involve $a$ or $b$ are unaffected. Wall crossings in $\Delta$ that contain $e = \{ a, b \}$ no longer exist in $\widetilde{\Delta}$. The possible changes in convexity come wall crossings with facets in $|\St_\Delta(a)| \cup |\St_\Delta(b)|$ or exactly 1 wall crossing away from it. They are given by the following cases: \\
			
			\begin{itemize}
				\item \textbf{Replacement of $a$ or $b$ by a ``contraction point'' $w$ interpolating $a$ and $b$ in the contraction vs. wall crossings in $\Delta$ using exactly one of $a$ or $b$. \\ }
				
				This does not involve a change in the simplicial complex structure and instead depends on the choice of contraction point $w$ (see hyperplanes bounding half-spaces intersected in Theorem \ref{contractpointspaceexpress}). The reason is that opposite off-wall vertices of a wall crossing in a flag simplicial complex are \emph{not} adjacent. Note that we want $|\St_\Delta(a)| \cup |\St_\Delta(b)|$ to be convex in order for convexity of wall crossings with respect to $w$ to be preserved. The wall relations involved in Proposition \ref{convsep} and Proposition \ref{ratwallconv} indicate that convexity of a wall crossing with respect to a particular on-wall vertex is defined by lying in some hyperplane. If $w$ is a linear interpolation point in the interior of $\overline{ab}$, it does \emph{not} lie on this hyperplane unless the entire line segment $\overline{ab}$ is contained in it. If $\overline{ab}$ is \emph{not} contained in this hyperplane, then the vertex $a$ or $b$ involved in the wall crossing from $\Delta$ uses up the single intersection point of the line containing $\overline{ab}$ and the hyperplane. This yields the statements about wall crossings inherited from $\Delta$ involving one of $a$ or $b$ since faces containing $e = \{ a, b \}$ no longer exist after contracting $e$. \\

				\item \textbf{Newly formed wall crossings among pairs of facets of $\Delta$ that did not share a wall that now share a wall in $\widetilde{\Delta}$. \\}

				This comes from pairs of consecutive wall crossings where the first wall crossing uses a facet $F$ only containing $a$ and not $b$ (respectively $b$ and not $a$) has the initial off-wall vertex $p$ replaced by $b$ to form a facet $G$ (respectively $a$) and the second wall crossing uses $a$ (respectively $b$) as the initial off-wall vertex that is replaced by a vertex $q$ to form a third facet $H$. The flatness of the resulting wall crossing using $\widetilde{F} \coloneq (F \setminus a) \cup w \in \widetilde{\Delta}$ and $\widetilde{H} \coloneq (H \setminus b) \cup w = (G \setminus \{ a, b \} ) \cup \{ q, w \}$ with respect to a vertex $m \in (F \cap G) \setminus a = (\widetilde{F} \cap \widetilde{H}) \setminus w$ is equivalent to $q \in \Span((\widetilde{F} \cap \widetilde{H}) \setminus m, p)$. If $m = w$, this is a property inherited from $\Delta$ before the edge contraction. Otherwise, having $m \ne w$ would mean that $w$ is included in the wall of $\widetilde{\Delta}$ and the choice does not only depend on the simplicial complex structure. The parameter would be the $k^d$-coordinate used for the contraction point and the intersections of hyperplanes bounding half-spaces from Theorem \ref{contractpointspaceexpress} would be a source of points. The analysis from wall crossings involving exactly one of $a$ or $b$ above indicates that flatness is difficult to maintain without additional conditions. So, we will focus on the $m = w$ case. \\
				
				In particular, we claim that the wall crossing for $F \cup G$ is \emph{not} flat with respect to $a$ (i.e. $b \notin \Span(F \setminus a)$). Assume that $q \in \Span((\widetilde{F} \cap \widetilde{H}) \setminus w, p) = \Span((F \cap G) \setminus a, p) = \Span(F \setminus a)$. Since the assumption on $q$ implies that $\Span(H \setminus b) = \Span(F \setminus a)$, the linear independence of the vertices of $H$ implies that $b \notin \Span(H \setminus b) = \Span(F \setminus a)$ and the wall crossing $F \cup G$ is \emph{not} flat. This is because $H \setminus b = (F \setminus \{ p, a \} ) \cup q$ and $q \in \Span(F \setminus a)$ would imply that $\Span(H \setminus b) \subset \Span(F \setminus a)$. The latter condition means that we have equality by dimension reasons.  The same reasoning using the hyperplane $\Span(H \setminus b)$ implies that the wall crossing for $G \cup H$ is \emph{not} flat with respect to $b$ (i.e. $a \notin \Span(H \setminus b)$). We can show this more directly using explicit wall relation coefficient signs. Using the wall relation for $G \cup H$ to replace $x_a$ by $x_q$ in the wall relation for $F \cup G$, we need the coefficient of $x_b$ in the resulting linear relation to be equal to $0$ while the coefficient of $x_q$ is strictly positive. Denoting the coefficient of $x_r$ in the wall relation for $F \cup G$ by $\alpha_r$ and the coefficient of $x_s$ in the wall relation for $G \cup H$ by $\beta_s$, this is equivalent to having \[ \alpha_b + \alpha_a \cdot -\frac{\beta_b}{\beta_a} = 0 \] and $\alpha_a \cdot -\frac{\beta_q}{\beta_a} > 0$. The latter condition follows from existing convexity assumptions since $\alpha_a < 0$ ($\alpha_a \le 0$ and $\alpha_a \ne 0$ from $\alpha_b > 0$) and $\beta_q, \beta_a > 0$. Thus, the quantitative condition for flatness of the resulting wall crossing in $\widetilde{\Delta}$ with respect to the contraction point $w$ is \[ \alpha_b + \alpha_a \cdot -\frac{\beta_b}{\beta_a} = 0 \Longleftrightarrow \alpha_b = \alpha_a \cdot \frac{\beta_b}{\beta_a} \Longleftrightarrow \frac{\alpha_a}{\alpha_b} = \frac{\beta_a}{\beta_b}  \] and the pair of wall crossings involved ($F \cup G$ and $G \cup H$) are \emph{not} flat with respect to $a$ and $b$. Note that $\alpha_b > 0, \alpha_a < 0, \beta_b < 0$, and $\beta_a > 0$. The inequalities $\alpha_b > 0$ and $\beta_a > 0$ follow from initial assumptions while $\alpha_a < 0$ and $\beta_b < 0$ follow from combining $\alpha_a \le 0$ and $\beta_b \le 0$ with the relation $\alpha_b + \alpha_a \cdot - \frac{\beta_b}{\beta_a} = 0$. \\  
				
				The wall crossings $F \cup G$ (respectively $G \cup H$) \emph{not} being flat with respect to $a$ (respectively $b$) is analogous to the new induced 4-cycles in $\widetilde{\Delta}$ coming from induced 5-cycles in $\Delta$, which do \emph{not} contain any induced 4-cycles. \\
				
				Recall from above that induced 5-cycles \emph{cannot} come from suspensions and must arise from edge subdivisions of induced 4-cycles or edge contractions of induced 6-cycles. Since suspensions only yield wall crossings which are flat with respect to all the on-wall vertices, there must be some kind of edge subdivision/contraction involved to construct the types of new wall crossings in $\widetilde{\Delta}$ that are \emph{not} counterparts of those in $\Delta$. In counterparts of existing wall relations, we note that edge subdivisions cannot produce new vertices inducing flat wall crossings (Proposition \ref{edgesubdivwallconv}). Since $\alpha_a < 0$ and $\beta_b < 0$ in the wall crossings $F \cup G$ and $G \cup H$, obtaining such wall crossings in $\Delta$ that lead to new flat wall crossings from $\widetilde{\Delta}$ are not induced by suspensions and must involve edge subdivisions/contractions in some way. 
			\end{itemize}

		\end{enumerate}

	\end{proof}

	\color{black}

	\section{Linear systems of parameters storing analogues of local convexity changes and replacing rational equivalence} \label{lsopcoordrat}

	For fans associated to toric varieties, the connection with local convexity starts with inner products in the rational equivalence relations. The wall relations determining local convexity data follow from intersecting these relations with torus-invariant curves (p. 300 of \cite{CLS}). For simplicial pseudomanifolds, we have replaced them PL homeomorphism-induced changes based on an initial ``base case'' (e.g. boundaries of cross polytopes for simplicial spheres). The end of Section \ref{presconpl} also suggests a heuristic of studying edge subdivisions of simplicial complexes with 1-skeleta covered by 4-cycles or flat wall crossings (Proposition \ref{bdry4cycleflat}). This has involved making choices about relationships between coordinate realizations of vertices of the simplicial complexes involved. In Section \ref{presconpl}, we assumed that the edge subdivisions and contractions composing the PL homeomorphism are realized by linear interpolations. However, the reasoning involved in a large part of the main results involved do not make use of this assumption (e.g. unexpected behavior discussed in Remark \ref{genreasonlsop}). \\
	
	Looking towards general realizations, we work with changes in linear systems of parameters, which we can interpret as \emph{relationships} between coordinate realizations of the vertices of a simplicial complex (Lemma \ref{lsopcond}, Remark \ref{lsopptconfig}, Remark \ref{lsopwallcoord}). They serve as a sort of middle ground between specific coordinate realizations and algebraic structures which might not store enough information on geometric properties that depend on the realizations in an important way. In this context, we make choices of linear systems of parameters that naturally store information on changes made under PL homeomorphisms and suspensions (Proposition \ref{lsopsuspedgediv} and Proposition \ref{edgecongenlsop}). It also clarifies the linear interpolation choice (e.g. for edge subdivisions) from Section \ref{presconpl} through examples such as Example \ref{crosspolylsop}. Afterwards, we apply this to study analogues of local convexity changes induced by PL homeomorphisms (Proposition \ref{flaglsoptrace}, Corollary \ref{plhomeomflagsign}). The changes observed are analogous to those of Section \ref{presconpl}. \\

	\color{black} 
	
	\subsection{Coordinate realizations and comparisons with algebraic structures}

	Recall that we took subdividing vertices and contraction points of edges to be realized by linear interpolation. We are looking towards local convexity-related information for more general choices of coordinate realizations. In order to figure out what properties are important we look at where we used these assumptions. Since it looks like they were often not important for key results, they seem to indicate that we can expect similar properties to hold for more general coordinate realizations. \\

	\begin{rem} \textbf{(Reasoning used in Section \ref{presconpl} and realizations) \\} \label{genreasonlsop}
	
		Although the results in Section \ref{presconpl} often assume subdividing vertices and contraction points are realized by linear interpolations, the reasoning behind them often does not require them. For example, this includes most of the proof of Theorem \ref{contractpointspaceexpress} (Lemma \ref{contractpointspace1}, Lemma \ref{contractpointspace2}, Lemma \ref{contractpointspace3}). This also applies to properties of intersections of convex sets on the boundary and connectivity properties of simplicial complexes involved (Proposition \ref{starunintspan}, Proposition \ref{starunionadj}, Proposition \ref{convwallspan}) used for restrictions on nearby convex contractible edges (Theorem \ref{convedconst}). Finally, the applications of Lemma \ref{contractpointspace1} for boundaries of cross polytopes and external edge subdivisions (Example \ref{contrcrosspolytop} and Theorem \ref{exsubdivcon}) inducing an empty contraction point space are independent of the linear independent assumption. Note that this property induces the emptiness of the contraction point space of the boundary of a cross polytope although more ``complicated'' conditions from Theorem \ref{contractpointspaceexpress} are satisfied. \\
	\end{rem}

	Keeping this mind, we consider examples and general structures of simplicial pseudomanifolds that have common properties with those used to study toric varieties. Our aim is to explore methods of analyzing combinatorial data describing how realizations of faces of simplicial complexes fit together. They will involve properties generalizing to simplicial complexes that do not necessarily come from simplicial fans associated to toric varieties. As a start, we look at relationships between realizations of vertices in boundaries of cross polytopes. \\

	\color{black}

	\begin{exmp} \label{crosssq0} \textbf{(Terms squaring to 0 in Chow rings and cross polytopes) \\} 
		To compare relevant algebraic structures, we note that the $h$-polynomial of the boundary of a $d$-dimensional cross polytope is $h(x) = (1 + x)^d$ and that it gives the even degree Betti numbers of the product $\mathbb{P}^1 \times \cdots \times \mathbb{P}^1$. In the Chow ring, we have that the square $x_\rho^2$ of each variable $x_\rho$ corresponding to a ray of the associated fan is equal to 0 (and that the converse holds by Proposition 1.13 on p. 18 -- 19 of \cite{Psig}). \\
	\end{exmp}
	
	This suggests looking at analogues of Chow rings of toric varieties for simplicial pseudomanifolds. In simplicial complexes, the substitute for rational equivalence in Chow rings of toric varieties comes from linear systems of parameters. They can be taken to be certain relationships between $k^d$-realizations of vertices. \\ 
	
	\begin{defn} \textbf{(Stanley--Reisner rings) \\} (Definition 5.1.2 on p. 209 of \cite{BH}, Definition 1.1 on p. 53 of \cite{St}) \\
		Given a simplicial complex $\Delta$ with vertex set $[n]$, its \textbf{Stanley-Reinser ring/face ring} is $k[x_1, \ldots, x_n]/I_\Delta$, where $I_\Delta$ is the ideal generated by monomials $x^G \coloneq \prod_{ \substack{p \in G \\ G \notin \Delta} } x_p$. 
	\end{defn}
	
	\begin{defn} (p. 81 of \cite{St}) \\
		Given a linear form $\theta = \sum_p \alpha_p x_p \in k[\Delta]_1$ and a face $F \in \Delta$, the \textbf{restriction} of $\theta$ to $F$, denoted $\theta|_F$ is defined by \[ \theta|_F \coloneq \sum_{p \in F} \alpha_p x_p. \]
	\end{defn}
	
	\begin{lem} \textbf{(Linear systems of parameters) \\} (Lemma 2.4 on p. 81 -- 82 of \cite{St}, p. 220 of \cite{BH}) \label{lsopcond}
		\begin{enumerate}
			\item Let $k[\Delta]$ be a face ring of Krull dimension $d$, and let $\theta_1, \ldots, \theta_d \in k[\Delta]_1$. Then $\theta_1, \ldots, \theta_d$ is a linear system of parameters (l.s.o.p.) for $k[\Delta]$ if and only if for every face $F \in \Delta$ (or equivalently, for every facet $F \in \Delta$), the restrictions $\theta_1|_F, \ldots, \theta_d|_F$ span a vector space of dimension equal to $|F|$.
			
			\item If $\theta_1, \ldots, \theta_d$ is an l.s.o.p. for $k[\Delta]$, then the quotient ring $k[\Delta]/(\theta_1, \ldots, \theta_d)$ is spanned (as a vector space over $k$) by the face monomials $x^F = \prod_{p \in F} x_p$ for $F \in \Delta$. \\
		\end{enumerate}
	\end{lem}
	
	\begin{rem} \textbf{(Choices of $k^d$-coordinates of vertices and linear systems of parameters) \\} \label{lsopptconfig}
		Rather than actual $k^d$-coordinates of the individual vertices $p \in V(\Delta) = [n]$ of a $(d - 1)$-dimensional Cohen--Macaulay simplicial complex $\Delta$, the linear system of parameters is about relations between the configurations of the points (analogous to common weighted averages/centers of mass). An interpretation involving stress spaces is considered in work of Lee (e.g. p. 394 of \cite{Lee}). The linear system of parameters forming a regular sequence implies that the (Krull) dimension is actually cut by 1 by each relation and a choice of $k^d$-coordinates of the vertices in $[n] \setminus F$ uniquely determines those of vertices of $F$. The non-redundancy of the linear relations can be thought of as linearly independent normal vectors of a collection of hyperplanes whose intersection has the expected dimension. \\
	\end{rem}
	
	The $k^d$-coordinate properties discussed above are discussed in the example below in the case of boundaries of cross polytopes. \\
	
	\begin{exmp} \textbf{(Boundaries of cross polytopes and suspensions) \\} \label{crosspolylsop}
		We note that linear systems of parameters give information on \emph{relationships} between $k^d$-coordinates of vertex realizations of $\Delta$ rather than specific coordinates themselves. They can also be interpreted as linear systems of parameters from repeated suspensions. For simplicity, we will describe this for boundaries of cross polytopes. Recall that the boundary of a cross polytope is formed by repeated suspensions. We focus on the $2$-dimensional case with $d = 3$. Writing the vertices as $x_{\pm 1}$, $x_{\pm 2}$, and $x_{\pm 3}$, Lemma \ref{lsopcond} implies that a possible linear system of parameters is given by the following linear forms:
		
		\begin{itemize}
			\item $\theta_1 \coloneq x_1 + x_{-1}$
			
			\item $\theta_2 \coloneq x_2 + x_{-2}$
			
			\item $\theta_3 \coloneq x_3 + x_{-3}$ 
		\end{itemize}
		
		Consider the Artinian reduction $A^\cdot(\Delta) = k[\Delta]/(\theta_1, \theta_2, \theta_3)$. To obtain coordinates of actual points in $k^3$, we need to make a \emph{choice} of coordinate representations of $x_1$, $x_2$, and $x_3$. The linear independence of these linear forms as elements of $k[x_1]$ is analogous to linear independence of points of a face of a simplicial complex if we take each face to form an actual simplex. A choice of points of coordinate representations of $x_1$, $x_2$, and $x_3$ determines ones for $x_{-1}$, $x_{-2}$, and $x_{-3}$. Finally, we note that the relations $\theta_1 \coloneq x_1 + x_{-1} = 0$, $\theta_2 \coloneq x_2 + x_{-2} = 0$, and $\theta_3 \coloneq x_3 + x_{-3} = 0$ are wall relations. Looking at the on-wall vertices, they indicate trivial conormal bundles as mentioned in Example \ref{crosssq0}. \\
		
	\end{exmp}

	However, allowing for arbitrary choices of linear systems of parameters does not give sufficient information to determine combinatorial data on how the simplicial complexes fit together. This occurs even after restricting to $(d - 1)$-dimensional balanced simplicial complexes (Definition 4.1 on p. 95 of \cite{St}), where the 1-skeleton has a proper $d$-coloring. In place of rational equivalence relations, the ``usual'' choice of linear system of parameters for a balanced simplicial complex also yields terms whose squares are equal to 0. \\
	
	\begin{prop} \label{colordegen} (Proposition 4.3 on p. 97 of \cite{St}) \\
		\vspace{-2mm}
		\begin{enumerate}
			\item Let $\Delta$ be a balanced simplicial complex of dimension $d - 1$ on the vertex set $V$. Define \[ \theta_i = \sum_{ \substack{ x \in V \\ \kappa(x) = i } } x \in k[\Delta] \] for $1 \le i \le d$. Note that $\theta_i$ is homogeneous of degree $e_i$ in $\mathbb{N}^d$. Then, $\theta_1, \ldots, \theta_d$ is a homogeneous system of parameters for $k[\Delta]$.
			
			\item Let $S = k[\Delta]/(\theta_1, \ldots, \theta_d)$. Then for every $x \in V$, we have $x^2 = x \theta_i = 0$ on $S$. Since we have a proper coloring, we also have $\theta_i^2 = 0$ for all $i \in [d]$. \\
		\end{enumerate}
	\end{prop}
	
	This indicates that a ``direct converse'' does \emph{not} seem to hold for balanced simplicial complexes when we allow arbitrary linear systems of parameters. In other words, we need some additional structure connected to information on $k^d$-coordinates used. To address this, we will focus on the linear systems of parameters themselves. In the case of toric varieties, we will specifically choose the rational equivalence relations. They give a middle ground between algebraic structures and specific choices of $k^d$-coordinate realizations. \\

	\begin{rem} \textbf{(Linear systems of parameters and wall relations vs. coordinate representations) \\} \label{lsopwallcoord}
		\begin{enumerate}
			\item Changes in linear systems of parameters from PL homeomorphisms of simplicial complexes via Proposition \ref{lsopsuspedgediv} and Proposition \ref{edgecongenlsop} give analogues of local convexity properties from wall relations of toric varieties. Some sign-related comments are in Remark \ref{edgesubdivconreal} and applied in later results.
			
			\item Fix a wall crossing involving two facets $F, F' \in \Delta$ sharing a wall. Given an initial linear system of parameters $\theta$ for the Cohen--Macaulay simplicial complex $\Delta$, the wall relations seem to be related to the space of linear systems of parameters $\widetilde{\theta}$ with $\widetilde{\theta}|_p = \theta|_p$ for all $p \in [n] \setminus (F \cup F')$. If we also a fix a choice of coordinates $x_p$ for $p \in F \cup F'$, then we have the relation \[ \sum_{p \in F \cup F'} (\widetilde{\theta}_{i, p} - \theta_{i, p}) x_p = 0 \] for all $i \in [d]$. This relation must be a multiple of the wall relation associated to $F \cup F'$. Since the off-wall rays must have nonzero coefficients of the same sign, there can only be at most one linear system of parameters of the form given in Part 2 of Proposition \ref{lsopsuspedgediv}. \\
			
			\item A sort of ``dual'' condition is related to $\ker \begin{pmatrix} \theta|_\tau & \theta|_a & \theta|_{a'} \end{pmatrix}$, where $a$ and $a'$ are off-wall vertices in the wall crossing associated to $F \cup F'$ and $\tau = F \cap F'$ is the wall. Suppose that we start with a linear system of parameters $\theta_1, \ldots, \theta_d$ and also fix $k^d$-coordinate representatives of $p \in [n] \setminus (F \cup F')$. Given these conditions, the space of possible choices of $k^d$-coordinates for $p \in F \cup F'$ is parametrized by this kernel since \[ \sum_{p \in F \cup F'} \theta_{i, p}(\overline{x}_p - x_p) = 0. \]
		\end{enumerate}
	\end{rem}

	\color{black}

	Building on Example \ref{crosspolylsop}, we will construct another example (Example \ref{edgelsopavg}) where we have a choice of linear system of parameters giving a direct connection to wall relations. It turns out that this choice is also compatible with linear interpolations. Before describing the linear systems of parameters involved, we will relate wall relations to restrictions of the Chow ring replacement $A^\cdot(\Delta)$ discussed above. \\

	\begin{rem} \textbf{(Wall crossing analogues from linear systems of parameters) \\} \label{wallanaloguelsop}
		Looking towards the behavior of the Chow ring replacements with respect to restriction relates the linear systems of parameters to analogues of wall relations. Coning (Theorem 7 on p. 396 -- 397 of \cite{Lee}) can be used to obtain the decomposition $A^k(\Delta) \cong A^k(\Ast_\Delta(p)) \oplus A^k(\St_\Delta(p)) \cong A^k(\Ast_\Delta(p)) \oplus x_p A^{k - 1}(\lk_\Delta(p))$ for $1 \le k \le d$ when $\dim \Delta = d - 1$. If $\Delta$ is flag (e.g. when it has convex wall crossings), we can apply this repeatedly to get $A^k(\Delta) \cong A^k(\Ast_\Delta(F)) \oplus A^k(\St_\Delta(F)) \cong A^k(\Ast_\Delta(F)) \oplus x_p A^{k - 1}(\lk_\Delta(p))$ for any $(m - 1)$-dimensional face $F \in \Delta$ with $m \le k$. Consider the second term of the direct sum. Suppose that we write the terms we multiply $x^F$ by as sums of the form $\alpha + \beta$ with $\alpha \in k[\lk_\Delta(F)]$ and $\beta \in (\theta_1, \ldots, \theta_d)$. Then, multiplication by $x^F$ filters out terms of $\beta$ that do \emph{not} form a face of $F$. If $m = d - 1$ (i.e. $F$ a wall/ridge), then the vertices that can appear in $\Delta$ are those of $F$ or the two other vertices of $\Delta$ attached to it. In some sense, such a coordinate relation for $\lk_\Delta(F)$ indicates that the restriction of the linear system of parameters to relevant faces gives an analogue of wall relations.  \\
	\end{rem}
	
	\color{black}

	Given the discussion above, we construct linear systems of parameters that naturally record PL homeomorphism information. This will be our main tool analyzing analogues of local convexity relations. Some discussion related to sign choices in Part 2 are in Remark \ref{edgesubdivconreal}. This is something we will keep in mind in later applications. \\

	\begin{prop} \textbf{(Linear systems of parameters vs. suspensions and edge subdivisions) \\} \label{lsopsuspedgediv} \\
		\vspace{-1mm}
	
		Let $\Delta$ be a $(d - 1)$-dimensional Cohen--Macaulay simplicial complex and $\theta_1, \ldots, \theta_d$ be a linear system of parameters for $\Delta$. 
		
		\begin{enumerate}
			\item Let $\Susp(\Delta)$ be the suspension $\Delta$ by vertices $p$ and $q$. Then $\theta_1, \ldots, \theta_d, x_p + x_q$ is a linear system of parameters for $\Susp(\Delta)$. Under such a linear system of parameters, we have that $x_p = -x_q$ and $x_p^2 = x_q^2 = 0$ in $A^\cdot(\Susp(\Delta))$. \\
			
			\item Let $\Delta'$ be the (stellar) subdivision of $\Delta$ with respect to an edge $e = \{ a, b \} \in \Delta$ and $v$ be the subdividing vertex.

			For any $\alpha, \beta \in k \setminus 0$, the following collection of linear forms $\ell_i$ for $i \in [d]$ is a linear system of parameters for $\Delta'$:

			\begin{itemize}
				\item For $q \ne v$, set $\ell|_q \coloneq \theta|_q$. In other words, we have $\ell_{i, q} = \theta_{i, q}$ for all $i \in [d]$. 
				
				\item Define $\ell|_v \coloneq \alpha \theta|_a + \beta \theta|_b$. In other words, we have $\ell_{i, v} = \alpha \theta_{i, a} + \beta \theta_{i, b}$ for all $i \in [d]$. If $\alpha + \beta = 1$, then $\ell|_v$ is a point of $k^d$ on the line containing $\theta|_a$ and $\theta|_b$ and it lies on the line segment $\overline{\theta|_a \theta|_b} \subset k^d$ if $0 < \alpha, \beta < 1$. If $\alpha + \beta = -1$, then $\ell|_v$ is such a point multiplied by $-1$. 
			\end{itemize}

		\end{enumerate}
	\end{prop}

	\begin{proof}
		\begin{enumerate}
			\item Recall that a collection of linear forms $\ell_1, \ldots, \ell_d$ forms a linear system of parameters for a $(d - 1)$-dimensional simplicial complex $S$ if and only if the restrictions $\ell_1|_F, \ldots, \ell_d|_F$ span a vector space of dimension $|F|$ for every face $F \in S$ (equivalently for every facet $F \in S$) (Lemma \ref{lsopcond}). \\
			
			Since $\theta_1, \ldots, \theta_d$ forms a linear system of parameters for $\Delta$, we have that $\theta_1|_F, \ldots, \theta_d|_F$ span a vector space of dimension $|F|$ for every face $F \in \Delta$. We note that $x_p + x_q$ restricts to $0$ for $F \in \Delta$. Thus, it remains to study the restrictions to the faces of $\Susp(\Delta)$ that contain $p$ or $q$. Without loss of generality, we can assume that they contain $p$. Since they are of the form $G = F \cup p$ for $F \in \Delta$, we have that $\theta_i|_G = \theta_i|_F$ for all $1 \le i \le d$. This implies that $\dim_k(\theta_1|_G, \ldots, \theta_d|_G) = |F|$. Since the $\theta_i$ do \emph{not} involve $x_p$ or $x_q$ for \emph{any} $i \in [d]$, this implies that $\dim_k(\theta_1|_G, \ldots, \theta_d|_G, (x_p + x_q)|_G = x_p) = |F| + 1 = |G|$. Thus, the linear forms $\theta_1, \ldots, \theta_d, x_p + x_q$ give a linear system of parameters for $\Susp(\Delta)$. \\
			
			\item As in Part 1, we will study the dimensions of spans of restrictions of linear forms to faces of the given simplicial complex. \\
			
			We can split the faces of $\Delta'$ into those that contain $v$ and those that do \emph{not} contain $v$. Note that $\lk_{\Delta'}(v) = \Susp_{a, b}(\lk_\Delta(e))$. \\ 
			
			Consider the restriction of the linear forms listed to faces $F \in \Delta'$ that do \emph{not} contain $v$. These correspond to the faces of $\Delta$ that do \emph{not} contain $e$. Since $\ell_i|_F = \theta_i|_F$ by definition and $\theta$ is a linear system of parameters for $\Delta$, the restrictions $\ell_1|_F, \ldots, \ell_d|_F$ span a vector space of dimension $|F|$. It remains to check the restrictions to faces $F \in \Delta'$ that contain $v$. \\

			If $F \in \Delta'$ and $F \ni v$, then we have that $F = G \cup v$, $F = G \cup \{ a, v \}$, or $F = G \cup \{ b, v \}$ for some $G \in \lk_\Delta(e)$. Since $\theta$ is a linear system of parameters for $\Delta$, we have that $\dim_k (\ell_1|_G, \ldots, \ell_d|_G) = |G|$ and $\dim_k (\ell_1|_{G \cup a}, \ldots, \ell_d|_{G \cup a}) = \dim_k (\ell_1|_{G \cup b}, \ldots, \ell_d|_{G \cup b}) = |G| + 1$. Consider the matrix $(\ell_{i, p})_{ \substack{ 1 \le i \le d \\ p \in H } }$ with $\ell_{i, j}$ equal to the coefficient of $x_p$ in $\ell_i$. Given a face $H \in \Delta$, these dimension statements can be restated as the row rank of this $d \times |H|$ matrix being equal to $|H|$. Since the row rank is equal to the column rank, this is equivalent to all the columns being linearly independent. In our context, this can be restated as $\dim_k \Span(\theta_1|_F, \ldots, \theta_d|_F) = \dim_k \Span(\theta|_p)_{p \in F}$ for any linear system of parameters $\theta$ of a $(d - 1)$-dimensional simplicial complex $F$. In order for the dimension to increase by 1 after appending adding the vertex $v$ to the faces $H$ considered, the column rank increases by 1 if and only if $\ell|_v = \alpha \theta|_a + \beta \theta|_b \notin \Span_k(\ell|_p)_{p \in H}$. Here, we consider $H = G$, $H = G \cup a$, or $H = G \cup b$ for $G \in \lk_\Delta(e)$. \\
			
			Suppose that $\ell|_v = \alpha \theta|_a + \beta \theta|_b  \in \Span_k(\ell|_p)_{p \in G}$ for some $G \in \lk_\Delta(e)$. Since $G \not\ni v$, we have that $\ell|_p = \theta|_p$ for all $p \in G$. This means that $\ell|_v = \alpha \theta|_a + \beta \theta|_b  \in \Span_k(\theta|_p)_{p \in G}$. However, this would contradict the assumption that \[ \dim \Span_k(\theta_1|_{G \cup e}, \ldots, \theta_d|_{G \cup e}) = \dim \Span_k(\theta|_p)_{p \in G \cup e } = |G| + 2 \] since the columns must be linearly independent in order for this to occur. Thus, we have that $\ell|_v \notin \Span_k(\ell|_p)_{p \in G}$ for all $G \in \lk_\Delta(e)$. \\
			
			If $\ell|_v = \alpha \theta|_a + \beta \theta|_b  \in \Span_k(\ell|_p)_{p \in G \cup a} = \Span_k(\theta|_p)_{p \in G \cup a} $ for some $G \in \lk_\Delta(e)$, then subtracting $\alpha \theta|_a$ from column $v$ of the matrix using an elementary column operation would imply that the column rank of $(\theta|_p)_{p \in G \cup e}$ is $\le |G| + 1$. However, this would contradict the assumption that $\theta$ is a linear system of parameters of $\Delta$ since the latter would imply that \[ \dim \Span_k(\theta|_p)_{p \in G \cup e} = |G| + 2. \] The same reasoning also implies that $\ell|_v \notin \Span_k(\theta|_p)_{p \in G \cup b}$ for all $G \in \lk_\Delta(e)$. Thus, all the columns of the matrix $(\ell_{i, p})_{ \substack{ 1 \le i \le d \\ p \in H } }$ are linearly independent if $H \in \Delta'$ and $H \ni v$ and $\ell_1, \ldots, \ell_d$ is a linear system of parameters for $\Delta'$.

		\end{enumerate}
	\end{proof}

	\color{black}

	The counterpart of this result for edge contractions is stated below. \\

	\begin{prop} \textbf{(Edge contractions and generic linear interpolations) \\} \label{edgecongenlsop} \\ 
		\vspace{-7mm}
		
		Let $\Delta$ be a Cohen--Macaulay $(d - 1)$-dimensional simplicial complex and $\theta_1, \ldots, \theta_d$ be a linear system of parameters for $\Delta$. Suppose that $k$ is an infinite field. \\
		
		Fix an edge $e = \{ a, b \} \in \Delta$. Let $\widetilde{\Delta}$ be the contraction of the edge $e$ in $\Delta$ and $w$ be the point that the edge is contracted to (i.e. $a$ and $b$ both identified with/converge to $w$). For all but finitely many choices of $\omega \in k \setminus 0$, the following collection of linear forms $\ell_i$ for $i \in [d]$ is a linear system of parameters for $\widetilde{\Delta}$:
		
		\begin{itemize}
			\item Given a vertex $q \ne a, b$ of $\Delta$, set $\ell|_q \coloneq \theta|_q$. In other words, we have that $\ell_{i, q} = \theta_{i, q}$ for all $i \in [d]$. 
			
			\item Define $\ell|_w \coloneq (1 - \omega) \theta|_a + \omega \theta|_b$. In other words, we have that $\ell_{i, w} = (1 - \omega) \theta_{i, a} + \omega \theta_{i, b}$ for all $i \in [d]$ and $\ell|_w$ is a point of $k^d$ on the line $L(\theta|_a, \theta|_b)$ spanned by $\theta|_a$ and $\theta|_b$. If $\omega \in (0, 1)$, then $\ell|_w$ is a point of $k^d$ on the line $L(\theta|_a, \theta|_b)$ spanned by $\theta|_a$ and $\theta|_b$ lying \emph{between} $\theta|_a$ and $\theta|_b$. \\
		\end{itemize}

	\end{prop}
	
	\begin{proof}
		Recall that $\dim_k \Span(\ell_1|_F, \ldots, \ell_d|_F) = \dim_k \Span(\ell|_p)_{p \in F}$ for any $d$-tuple of linear forms $\ell_1, \ldots, \ell_d$ on $x_1, \ldots, x_n$ (see proof of Proposition \ref{lsopsuspedgediv}). In particular, this implies that \[ \dim_k \Span(\ell_1|_F, \ldots, \ell_d|_F) = |F| \] if and only if $(\ell|_p)_{p \in F}$ has linearly independent columns. \\ 
		
		We will split the facets of the edge contraction $\widetilde{\Delta}$ into those that contain $w$ and those that do \emph{not} contain $w$. If a facet $F \in \widetilde{\Delta}$ does \emph{not} contain $w$, then $F \in \Delta$ and does \emph{not} contain $a$ or $b$. Since $\ell|_q = \theta|_q$ for $q \ne a, b$, we have that $\ell_i|_F = \theta_i|_F$ for each $i \in [d]$. This implies that $\dim_k \Span(\ell_1|_F, \ldots, \ell_d|_F) = \dim_k \Span(\theta_1|_F, \ldots, \theta_d|_F) = |F| = d$ since $\theta$ is a linear system of parameters for $\Delta$ and $F \in \Delta$. \\
		
		Now suppose that $F \in \widetilde{\Delta}$ is a facet that \emph{does} contain $w$. Since $\widetilde{\Delta}$ is the contraction of the edge $e = \{ a, b \} \in \Delta$, we have that $\lk_{\widetilde{\Delta}}(w) = \lk_\Delta(a) \cup \lk_\Delta(b)$. This means that $F = G \cup w$ for some facet $G \in \lk_\Delta(a)$ or $G \in \lk_\Delta(b)$. Note that $G \in \Delta$ in both cases. Since $\theta$ is a linear system of parameters for $\Delta$ and $\ell|_q = \theta|_q$ for a vertex $q \ne a, b$ of $\Delta$, it suffices to show that considering the restrictions to $F = G \cup w$ instead of $G$ increases the (column) rank of the restriction matrix $(\ell|_p)_{p \in H}$ by 1. If $G \in \lk_\Delta(a)$, then $\theta|_a \notin \Span_k (\ell|_p)_{p \in G} = \Span_k (\theta|_p)_{p \in G}$ since $G \cup a \in \Delta$ and $\theta$ is a linear system of parameters for $\Delta$. In particular, this means that the line $L(\theta|_a, \theta|_b) \not\subset \Span_k (\ell|_p)_{p \in G} = \Span_k (\theta|_p)_{p \in G}$ and $L(\theta|_a, \theta|_b) \cap \Span_k (\ell|_p)_{p \in G} = L(\theta|_a, \theta|_b) \cap \Span_k (\theta|_p)_{p \in G}$ is a finite subset of $k^d$. Since $\lk_\Delta(a)$ has finitely many facets, this also implies that \[ \mathcal{S}_a \coloneq \bigcup_{G \in \lk_\Delta(a)} \left( L(\theta|_a, \theta|_b) \cap \Span_k (\ell|_p)_{p \in G} \right) = L(\theta|_a, \theta|_b) \cap \left( \bigcup_{G \in \lk_\Delta(a)} \Span_k (\ell|_p)_{p \in G} \right) \] is finite. \\
		
		Using the same reasoning, we have that \[ \mathcal{S}_b \coloneq \bigcup_{H \in \lk_\Delta(b)} \left( L(\theta|_a, \theta|_b) \cap \Span_k (\ell|_p)_{p \in H} \right) = L(\theta|_a, \theta|_b) \cap \left( \bigcup_{H \in \lk_\Delta(b)} \Span_k (\ell|_p)_{p \in H} \right) \] is finite. \\
		
		If $\ell|_w \in L(\theta|_a, \theta|_b) \setminus (\mathcal{S}_a \cup \mathcal{S}_b)$, this means that $\ell|_v \notin \Span_k (\ell|_p)_{p \in A}$ for any facet $A \in \lk_\Delta(a) \cup \lk_\Delta(b)$ and $\ell$ is a linear system of parameters of the contraction $\widetilde{\Delta}$ of $\Delta$ with respect to the edge $e = \{ a, b \} \in \Delta$. Since every point of the line $L(\theta|_a, \theta|_b)$ is of the form $\theta|_a + \omega (\theta|_b - \theta|_a) = (1 - \omega) \theta|_a + \omega \theta|_b$ for some $\omega \in k$, we have that $\ell|_w = (1 - \omega) \theta|_a + \omega \theta|_b$ for some ``generic'' $\omega \in k \setminus 0$. If $\omega \in (0, 1)$, then $\ell|_w$ lies between $\theta|_a$ and $\theta|_b$ on the line $L(\theta|_a, \theta|_b)$. 
		
	\end{proof}

	We specialize the results above to boundaries of cross polytopes. The linear of system of parameters in Proposition \ref{lsopsuspedgediv} will be compatible with the ``usual'' linear interpolation coordinates for subdividing vertices. It gives a concrete example where an appropriate choice of linear system of parameters gives rise to actual wall relations (see Remark \ref{wallanaloguelsop}). Note that is not possible for contractions of edges of boundaries of cross polytopes to preserve convexity of wall crossings (Example \ref{contrcrosspolytop}).  \\

	\begin{exmp} \textbf{(Edge subdivisions of boundaries of cross polytopes) \\} \label{edgelsopavg}

		Using the initial vertices $\pm 1, \pm 2, \pm 3$ from Example \ref{crosspolylsop}, we will write $12$ to mean the vertex subdividing the edge $\{ 1, 2 \} \in \Delta$. Let $\Delta'$ be the (stellar) subdivision of $\Delta$ with respect to the edge $\{ 1, 2 \}$. Suppose that $\alpha = \beta = -\frac{1}{2}$ in Part 2 of Proposition \ref{lsopsuspedgediv}. Then, the linear system of parameters from Example \ref{crosspolylsop} is now modified to $\ell_1 = x_1 + x_{-1} - \frac{1}{2} x_{12}$, $\ell_2 = x_2 + x_{-2} - \frac{1}{2} x_{12}$, and $\ell_3 = x_3 + x_{-3}$. \\
		
		As mentioned in Remark \ref{lsopptconfig}, we need to make an initial choice of coordinates in order to determine the rest of them. We will take $x_1, x_2, x_3$ to each be assigned to the elements of some basis of $k^3$. Suppose that $x_{12} = \frac{x_1 + x_2}{2}$. We now consider the effect of the edge subdivision on wall relations using this linear system of parameters and choice of coordinates for the subdividing point $x_{12}$. \\
		
		The facets of $\Delta'$ containing $12$ are those of the form $\{ 1, 12, \pm 3 \}$ or $\{ 12, 2, \pm 3 \}$. The wall crossings starting with such facets take one of the following forms: \\
		
		\begin{itemize}
			\item Replacing $3$ by $-3$ or $-3$ by $3$. Both facets in the wall crossing contain $12$. The associated wall relation is \[ x_3 + x_{-3} = 0. \] The coefficients of the on-wall vertices ($1$ and $12$ or $12$ and $2$) are 0. This is the same wall relation as the ones among facets of $\Delta$ replacing $3$ by $-3$. \\
			
			\item Replacing $1$ by $2$ or $2$ by $1$. Both facets in the wall crossing contain $12$. The chosen relation $x_{12} = \frac{x_1 + x_2}{2}$ implies that the wall relation is \[ x_1 + x_2 - 2 x_{12} = 0. \] Note that the coefficient of $x_3$ or $x_{-3}$ (depending on the wall crossing) is 0 in both such wall crossings. This is a new type of wall crossing not considered in $\Delta$. \\
			
			\item Replacing $12$ in $\{ 1, 12, \pm 3 \}$ by $-2$. The second facet in the wall crossing does \emph{not} contain the subdividing vertex $12$. In this wall crossing, we note that $12$ plays the role of $2$ in $\Delta$ in the facet we start with. Under our chosen linear system of parameters, we have that $x_2 + x_{-2} - \frac{1}{2} x_{12} = 0$ from $\ell_2 = 0$. The choice $x_{12} = \frac{x_1 + x_2}{2}$ implies that $x_2 = 2 x_{12} - x_1$. Substituting this into the relation $x_2 + x_{-2} - \frac{1}{2} x_{12} = 0$, we have that $2 x_{12} - x_1 + x_{-2}  - \frac{1}{2} x_{12} = 0$ and the wall relation is \[ \frac{3}{2} x_{12} + x_{-2} - x_1 = 0. \] As in the previous case, the coefficient of $x_3$ or $x_{-3}$ is 0 for such wall crossings. This is an analogue of wall crossings in $\Delta$ replacing $2$ by $-2$ with the wall relation $x_2 + x_{-2} = 0$. In some sense, the wall crossing after the edge subdivision is ``more convex'' with respect to the on-wall vertex $1$ (see Lemma 14-1-7 on p. 421 of \cite{Mat})  \\
			
			\item Replacing $12$ in $\{ 12, 2, \pm 3 \}$ by $-1$. Then, the same reasoning as the previous case yields the wall relation is \[ \frac{3}{2} x_{12} + x_{-1} - x_2 = 0. \] The coefficient of $x_3$ or $x_{-3}$ is 0 for such wall crossings. This is an analogue of wall crossings in $\Delta$ replacing $1$ by $-1$ with the wall relation $x_1 + x_{-1} = 0$. As in the previous case, the wall crossing is ``more convex'' with respect to the on-wall vertex $2$ (see Lemma 14-1-7 on p. 421 of \cite{Mat}). \\
		\end{itemize}

	\end{exmp}

	\begin{rem} \textbf{(Realizations of edge subdivisions and contractions) \\} \label{edgesubdivconreal}
		We note that the increased availability of convex wall crossings after edge subdivisions using linear interpolation as in Proposition  \ref{edgesubdivwallconv} is consistent with preservation of flagness of simplicial complexes after edge subdivisions. Note that we can show the linear system of parameters in Proposition \ref{lsopsuspedgediv} can be made consistent with linear interpolations (see Example \ref{edgelsopavg}). Since locally convex simplicial complexes are flag (Proposition 5.3 on p. 279 -- 280 of \cite{LR}), we can think of going further along edge subdivisions among flag simplicial pseudomanifolds PL homeomorphic to the starting one in terms of ``increased convexity''. This is something we take into account in the case of using negative weights to subdividing vertices in Proposition \ref{lsopsuspedgediv} and later applications of it. Since edge contractions make it more difficult to preserve convexity and are inverses of edge subdivisions when they belong to a subdivided edge, we will assign positive weights. \\
		
		This can be described more concretely using ideas from the computations in Example \ref{edgelsopavg}, where the initial linear system of parameters consists of actual wall relations. In particular, the substitution of a linear interpolation expression of the form $x_v = (1 - \eta) x_a + \eta x_b$ with $\eta \in (0, 1)$ for the subdividing vertex $ab$ of an edge $e = \{ a, b \}$ yields coefficient changes of the new wall relations that are compatible with the preservation of convexity discussed above and in Proposition \ref{edgesubdivwallconv}. This suggests that negative weights assigned to subdividing vertices in Proposition \ref{lsopsuspedgediv} are compatible with the conditions of having a wall relation and preservation of convexity of wall crossings. Considering edge contractions as a sort of ``inverse'' to edge subdivisions would then mean assigning negative signs to contraction points. \\
	\end{rem}
	
	\color{black}
	
	\subsection{PL homeomorphism-induced changes on analogues of wall relations}

	\color{black}

	We now apply the framework involving linear systems of parameters and related signs from the previous subsection (e.g. Proposition \ref{lsopsuspedgediv} and Proposition \ref{edgecongenlsop}, Remark \ref{edgesubdivconreal}) to analogues of local convexity changes. Heuristically, one can consider realizations of simplicial spheres PL homeomorphic to cross polytopes to record changes in local convexity. \\
	
	\color{black}
	
	\begin{prop} \label{flaglsoptrace}
	
		Let $\Delta$ be a simplicial sphere PL homeomorphic to the boundary of a cross polytope. Suppose that we use the linear systems of parameters in Proposition \ref{lsopsuspedgediv} and Proposition \ref{edgecongenlsop} for edge subdivisions and contractions respectively. The change in linear systems of parameters from PL homeomorphisms is described below.

		\begin{enumerate}
			\item The coefficients of the variable $x_p$ associated to a vertex $p \in V(\Delta)$ are weighted averages of coefficients of suspension and edge subdivision points. These weights are products of weights from ``layers'' of edge contractions the vertices are (nested) inputs for. For edge subdivisions and contractions, we will take these weights to lie in $(-1, 0)$ and $(0, 1)$ respectively. This is motivated by ``increased convexity'' and negative coefficients of subdividing vertices from edge subdivisions in Proposition \ref{edgesubdivwallconv} and Example \ref{edgelsopavg} discussed in Remark \ref{edgesubdivconreal}. \\
			
			\item If all the weights in Part 1 are equal to $\pm \frac{1}{2}$, then the coefficient of $x_p$ is the ``weighted difference'' of suspension points and subdivision points which are either equal to $p$ or contraction points leading to $p$. The weight of a particular input vertex $m$ is determined by the number of edge contractions involving $m$ or a (nested) contraction of it. Note that the weights can be set arbitarily close to $\pm \frac{1}{2}$ since there are only finitely many weights to omit for any particular variable (Proposition \ref{edgecongenlsop}). By Proposition \ref{lsopsuspedgediv} and Proposition \ref{edgecongenlsop}, we can take the weights in $(0, 1)$ and $(-1, 0)$ to be arbitarily close to $\frac{1}{2}$ and $-\frac{1}{2}$ respectively. \\
		\end{enumerate}
	\end{prop}

	\begin{proof}

		Vertices $p \in V(\Delta)$ are initially introduced either via suspensions or edge subdivisions. Suspensions produce new linear form of $\ell_k : x_k + x_{-k}$ (and coefficient 1) and edge subdivisions give $-1$ multiplied by the weighted average of the coefficients of the vertices bounding the subdivided edge when $p$ is introduced. In the latter case, the subdividing vertex $x_v$ with respect to an edge $e = \{ a, b \}$ appears in linear forms where $x_a$ or $x_b$ is used. This gives the actual coordinate of the vertex if it is \emph{not} the contraction point of an edge contraction. \\

		Suppose that $p$ is the contraction point of some edge contraction. Working backwards, the coefficient of a vertex $p$ in the very last step of the edge subdivisions/contractions from a cross polytope is a weighted average of the pair of vertices merged to form it. These vertices in turn are either introduced once via a suspension or edge subdivision (as the subdividing vertex). In general, we can trace the vertices involved in edge contractions. Tracing backwards through ``layers'', we see that $\theta|_p$ is a weighted average of initial suspension and subdividing vertices used to form the inputs of edge contractions forming inputs leading to the current vertex $p$. For example, having $\theta|_w = (1 - \eta) \theta|_a + \eta \theta|_b$ for some $\eta \in (0, 1)$ and $\theta|_a = (1 - \omega) \ell_r + \omega \ell_s$ for some $\omega \in (0, 1)$ with $a$ as the contraction point of an edge $\{ r, s \} \in \Delta$ and $\ell$ as the linear system of parameters just before contraction of the edge $e = \{a, b\}$ would mean that $\theta|_w = (1 - \eta) (1 - \omega) \ell|_r + (1  - \eta) \omega \ell|_s + \eta \theta|_b$. The weights associated to suspension vertices are positive while those associated to subdividing vertices of edges are negative. \\
		
		If all the weights in $(0, 1)$ are set equal to $\frac{1}{2}$, this amounts to weighting the input vertices by $\frac{1}{2^m}$, where $m$ is the number of edge contractions the vertex is an input for (either directly or as a ``nested'' input) before multiplying by $1$ for a suspension vertex and $-1$ for a subdividing vertex. The linear forms in which the variables appear can also be traced backwards. In particular, the variable $x_p$ appears in the linear forms where the (nested) input vertices for edge contractions leading to the vertex $p$ (or just $p$ itself if $p$ not a contraction point) occur. 
		
	\end{proof}

	\begin{cor} \textbf{(``Expected'' convexity properties and edge subdivisions/contractions) \\} \label{plhomeomflagsign}
		Suppose that $\Delta$ is a flag simplicial complex. Tracking signs and restricting to vertices attached to a particular codimension 1 face/wall of $\Delta$, Proposition \ref{flaglsoptrace} yields analogues of wall relations. The signs of the vertices $p$ involved are then yield counterparts (non)convexity of the wall crossing with respect to $p$ corresponding to nonpositive/nonnegative coordinates in wall relations. \\
	\end{cor}
	
	\begin{proof}
		Most of this was discussed in Remark \ref{wallanaloguelsop}. Let $F \in \Delta$ be a face of $\Delta$. Recall that each $A^k(\Delta)$ ($0 \le k \le d$) decomposes into a direct sum of the kernel of multiplication by $x^F$ ($A^k(\Ast_\Delta(F))$) and the image of multiplication by $x^F$ ($A^k(\St_\Delta(F))$). Suppose that $F \in \Delta$ is a codimension 1 face/wall. Then, the restrictions of linear systems of parameters to such faces $F$ yield wall relations. The statement then follows form applying p. 420 -- 421 of \cite{Mat} and Proposition \ref{flaglsoptrace}.
	\end{proof}
	
	The same reasoning as above can be used to analyze \emph{changes} in linear systems of parameters and convexity properties of PL-homeomorphism classes of simplicial complexes formed by repeated suspensions. Some related material on connections between wall relations and linear system of parameters is in Remark \ref{lsopwallcoord}. \\

	\begin{cor} \textbf{(Generalizations for PL homeomorphisms to (repeated) suspensions) \\} \label{genplhomeom}
		Suppose that $\Delta$ is a simplicial complex PL-homeomorphic to a repeated suspension of some simplicial complex $A$. Assume that we use the linear systems of parameters Proposition \ref{lsopsuspedgediv} and Proposition \ref{edgecongenlsop}. Then, the \emph{change} in linear system of parameters and convexity properties from this repeated suspension of $A$ to get to $\Delta$ is given by Proposition \ref{flaglsoptrace} and Corollary \ref{plhomeomflagsign}.
	\end{cor}
	
	\begin{proof}
		This follows from the fact that PL homeomorphic simplicial complexes are connected by a sequence of (stellar) edge subdivisions and their inverses (Corollary [10:2d] on p. 302 of \cite{Alex} and Corollary 4.1 on p. 75 of \cite{LN}). \\
	\end{proof}

	\begin{rem} \textbf{(Comments on parameters involved) \\}
		A direct analogue of the results above for the convexity context in the previous section can be obtained using direct substitution (also see Remark \ref{edgesubdivconreal}). The suspensions will yield linear forms $\ell_k = x_k + x_{-k}$ as in the previous section and the edge subdivisions will come from direct substitutions $x_v$ in place of $x_a$ or $x_b$ of relations $x_a + x_b - 2 x_v = 0$. For example, the variable $x_a$ will be replaced by $2x_v - x_b$. How we treat sign relations in wall relations will depend on whether the replaced vertex is an off-wall vertex or an on-wall vertex. In the case of edge contractions, we will use the linear system of parameters in Proposition \ref{edgecongenlsop}. This would yield coefficients/changes in coefficients in wall relations that are weighted averages. However, the input signs are not only a function of edge subdivisions/contractions but also whether the vertices introduced yield off-wall or on-wall vertices of wall crossings.   
	\end{rem}

\end{document}